\documentclass[10pt,draftcls,onecolumn]{IEEEtran}
\usepackage{ifpdf}
\usepackage{cite}
\if CLASSINFOpdf
  \usepackage[pdftex]{graphicx}

\else

\fi

\usepackage{amsmath,amssymb,amsfonts}
\usepackage{latexsym}
\usepackage{psfrag}
\usepackage{graphicx}
\usepackage[all]{xy}
\usepackage{algorithmic}
\usepackage{subfigure}
\usepackage{indentfirst}
\usepackage{epsfig}
\usepackage{epstopdf}
\usepackage{graphics}
\usepackage{array}
\usepackage{stfloats}
\usepackage{enumerate}
\usepackage{xcolor}
\usepackage{flushend}
\usepackage{times}

\newtheorem{theorem}{\textbf{Theorem}}
\newtheorem{definition}{\textbf{Definition}}
\newtheorem{corollary}{\textbf{Corollary}}

\newtheorem{proposition}{\textbf{Proposition}}
\newtheorem{assumption}{\textbf{Assumption}}

\newtheorem{lemma}{\textbf{Lemma}}
\newtheorem{remark}{\textbf{Remark}}

\newcommand{\veq}{\mathrel{\phantom{=}} }
\newcommand{\tabincell}[2]{\begin{tabular}{@{}#1@{}}#2\end{tabular}}
\DeclareMathOperator*{\argmin}{arg\,min}
\DeclareMathOperator*{\bigtimes}{\text{\large $\times$}} 

\long\def\comment#1{}

\graphicspath{{figures/}}

\begin{document}

\title{Distributed Algorithm Over Time-Varying Unbalanced Topologies for Optimization Problem Subject to Multiple Local Constraints
\thanks{This work was supported in part by the National Natural Science Foundation of China under Grant No.~61673107, the General joint fund of the equipment advance research program of Ministry of Education under Grant No.~6141A020223, and the Jiangsu Provincial Key Laboratory of Networked Collective Intelligence under Grant No.~BM2017002. \emph{(Corresponding author: Wenwu Yu.)}}
}

\author{Hongzhe Liu, Wenwu Yu, Guanghui Wen, and Wei Xing Zheng,~\IEEEmembership{Fellow, IEEE}
\thanks{H. Liu, W. Yu, and G. Wen are with the Department of Systems Science, School of Mathematics, Southeast University, Nanjing 210096, China (e-mail: 230169408@seu.edu.cn; wwyu@seu.edu.cn; wenguanghui@gmail.com).
G. Wen is also with the School of Engineering, RMIT University, Melbourne VIC 3001, Australia.}
\thanks{W. X. Zheng is with the School of Computer, Data and Mathematical Sciences, Western Sydney University, Sydney, NSW 2751, Australia (e-mail: w.zheng@westernsydney.edu.au).}
}

\maketitle

\date{}

\begin{abstract}
This paper studies the distributed optimization problem with possibly nonidentical local constraints, where its global objective function is composed of $N$ convex functions. The aim is to solve the considered optimization problem in a distributed manner over time-varying unbalanced directed topologies by using only local information and performing only local computations. Towards this end, a new distributed discrete-time algorithm is developed 
by synthesizing the row stochastic matrices sequence and column stochastic matrices sequence analysis technique.
Furthermore, for the developed distributed discrete-time algorithm, its convergence property to the optimal solution as well as its convergence rate are established under some mild assumptions. Numerical simulations are finally presented to verify the theoretical results.
\end{abstract}

\begin{IEEEkeywords}
Convex optimization, distributed algorithm, time-varying unbalanced directed topologies, nonidentical local constraints.
\end{IEEEkeywords}

\section{Introduction}

In recent decades, with the rapid development of various communication technologies, much more attention has been focused upon distributed control and optimization of multi-agent systems.
As an important topic within this context, distributed optimization has recently received an increasing research interest due to its potential applications in numerous fields, such as formation control of robot networks, source location of sensor networks, machine learning, power grid dispatching, etc.

In the past decade, many interesting results have been reported in the literature, where distributed optimization algorithms including continuous-time algorithms and discrete-time algorithms have been developed respectively for solving such optimization problems under different circumstances.
The results on various distributed continuous-time algorithms can be found in \cite{Wang2011CDC,Gharesifard2014TAC,Kia2015Aut,Liang2019Aut,Liu2015TAC,Shi2013TAC,Qiu2016Aut,Lin2012CDC,
Lin2017TAC,Yan2014NN,Yi2014CCC,Yi2015SCL,Liu2017TNNLS,Liu2013TNNLS,Yang2017TAC} and references therein.
This paper mainly focuses on designing a distributed discrete-time algorithm.

\begin{table}\label{Tab1}
\centering
\caption{Overview of Distributed Discrete-time Optimization Algorithms}
\begin{tabular}{|@{\hspace{0.3em}}c@{\hspace{0.3em}}|@{\hspace{0.3em}}c@{\hspace{0.3em}}%
                |@{\hspace{0.3em}}c@{\hspace{0.3em}}|@{\hspace{0.3em}}c@{\hspace{0.3em}}|c@{\hspace{0.3em}}|}
  \hline
  References & Constraints & \tabincell{c}{Linear\\ convergence rate} & \tabincell{c}{Unbalanced\\ topologies} & \tabincell{c}{Time-varying\\ topologies}\\
  \hline\hline
  \cite{Nedic2009TAC} & $\bigtimes$ & $\bigtimes$ & $\bigtimes$ & $\surd$ \\
  \cite{Nedic2010TAC,Yuan2016SIAM}& $\surd$ & $\bigtimes$ & $\bigtimes$ & $\surd$ \\
  \cite{Shi2017SIAM} & $\bigtimes$ & $\surd$ & $\bigtimes$ & $\surd$ \\
  \cite{Lei2016SCL,Liu2017TAC} & $\surd$ & $\bigtimes$ & $\bigtimes$ & $\bigtimes$ \\
  \cite{Nedic2015TAC} & $\bigtimes$ & $\bigtimes$ & $\surd$ & $\surd$ \\
  \cite{Nedic2017SIAM} & $\bigtimes$ & $\surd$ & $\surd$ & $\surd$ \\
  \cite{Liang2020TAC,Gu2020NA} & $\surd$ & $\bigtimes$ & $\surd$ & $\surd$ \\
  \cite{Xin2015SCL,Pu2018CDC,Xin2020TAC} & $\bigtimes$ & $\surd$ & $\surd$ & $\bigtimes$ \\
  \cite{Saadatniaki2018} & $\bigtimes$ & $\surd$ & $\surd$ & $\surd$ \\
  \cite{Mai2016ACC,Mai2019Aut} & $\surd$ & $\bigtimes$ & $\surd$ & $\bigtimes$ \\
  This work & $\surd$ & $\bigtimes$ & $\surd$ & $\surd$ \\
  \hline
\end{tabular}
\smallskip\\
\footnotesize{$\surd$ means that this feature is involved}\\
\footnotesize{\hspace*{1.2em}$\bigtimes$ means that this feature is uninvolved}
\vspace{-2ex}
\end{table}

In the context of distributed discrete-time optimization, the distributed discrete-time algorithms were given in \cite{Nedic2009TAC,Nedic2010TAC,Yuan2016SIAM, Shi2017SIAM,Lei2016SCL,Liu2017TAC}
to solve the optimization problems over balanced topologies under different scenarios.
Later, the result was extended to the case with time-varying unbalanced directed topologies in \cite{Nedic2015TAC} and \cite{Nedic2017SIAM}, where the distributed algorithms were designed for the unconstrained optimization problem with the column stochastic matrices used.
Subsequently, built on the work in \cite{Nedic2015TAC}, a distributed algorithm was developed in \cite{Liang2020TAC} over time-varying unbalanced directed topologies by integrating the dual averaging method into the push-sum mechanism, which can solve the optimization problem with a global closed convex set constraint.
Moreover, the optimization problem with coupled linear equality constraint was addressed in a similar way in \cite{Gu2020NA}.
Additionally, based on \cite{Nedic2017SIAM}, the novel push-pull algorithms with both row stochastic matrices and column stochastic matrices were proposed in \cite{Xin2015SCL,Pu2018CDC,Xin2020TAC,Saadatniaki2018} for the unconstrained optimization problem over unbalanced directed graphs which can be static or time-varying.
However, those results cannot be applied in the constrained case.
Recently, in \cite{Mai2016ACC} and \cite{Mai2019Aut}, some attempts were made at adopting row stochastic iterative matrices for dealing with the optimization problems over static unbalanced directed topologies as the row stochastic property is very standard and easy to be satisfied in the distributed setting.
In \cite{Mai2016ACC}, an optimization problem with identical compact set constraint was considered and an algorithm based on the mechanism of estimating the left eigenvector associated with the eigenvalue $1$ was proposed.
The result in \cite{Mai2016ACC} was further extended in \cite{Mai2019Aut} to a more general case where nonidentical general closed convex set constraints were involved.
To get a clear overview of the state of the art, Table~\ref{Tab1} summarizes a detailed comparison of the aforementioned results in  \cite{Nedic2009TAC,Nedic2010TAC,Nedic2017SIAM,Lei2016SCL,Liu2017TAC,Nedic2015TAC,Liang2020TAC,Gu2020NA,Xin2015SCL,Pu2018CDC,Xin2020TAC,%
Saadatniaki2018,Mai2016ACC,Mai2019Aut}.

Although the good result was achieved in \cite{Liang2020TAC}, it is worth mentioning that the operation $\argmin_{x\in X}$ was involved in the algorithm designed in \cite{Liang2020TAC}, with $X$ being the constraint set.
Thus, each iteration of the designed algorithm therein is required to solve a subproblem $\argmin_{x\in X}$ first, which could be a numerically demanding task.
Additionally, it should be noted that only the optimization problem with a global closed convex set constraint was considered in \cite{Liang2020TAC}, but real-world optimization problems necessarily involve more general constraints quite often.
Motivated by this, we aim to solve the optimization problem with more general constraints over time-varying unbalanced directed topologies in this paper, for which a desirable distributed discrete-time algorithm will be designed without involving any in-built operation $\argmin$.

Specifically, the optimization problem with $N$ nonidentical convex inequality constraints and compact set constraints is considered in this paper.
A new distributed discrete-time algorithm is proposed over time-varying unbalanced directed topologies and a rigorous analysis of its convergence to the optimal solution of the considered problem is made under some mild assumptions.
Moreover, a detailed analysis of the convergence rate is also shown. Thus, the major contribution of this paper is that a challenging problem is successfully solved by a novel distributed discrete-time algorithm.
In detail, first, although both row stochastic matrices and column stochastic matrices are involved, the designed distributed discrete-time algorithm herein is essentially different from the algorithms proposed in \cite{Xin2015SCL,Pu2018CDC,Xin2020TAC,Saadatniaki2018} since different analysis methods are adopted.
As a result, the optimization problem with multiple nonidentical constraints can be solved over time-varying unbalanced directed topologies in this paper, which, however, could not in \cite{Xin2015SCL,Pu2018CDC,Xin2020TAC,Saadatniaki2018}. Second, the strong convexity of local objective functions is not necessarily required in this paper, while this is the key to the convergence analysis made in \cite{Xin2015SCL,Pu2018CDC,Xin2020TAC,Saadatniaki2018}.
Third, unlike the results given in \cite{Mai2016ACC} and \cite{Mai2019Aut}, auxiliary variables are introduced in the present paper to estimate the gradients of the global objective function rather than the left eigenvector for canceling out the gradients error of the global objective function caused by the asymmetry of time-varying unbalanced directed topologies. Thus, the dimensions of auxiliary variables as well as the total dimensions of all generated vectors associated with the proposed algorithm are considerably lower than those in \cite{Mai2016ACC} and \cite{Mai2019Aut} when large-scale networks are involved, which is of great importance in real applications.
Fourth, compared to the algorithm designed in \cite{Liang2020TAC}, the optimization problem with much more general constraints can be resolved by the algorithm proposed in this paper, where the in-built operation $\argmin$ is not involved at all.
Finally, this paper potentially provides an efficient mechanism for designing distributed algorithms to solve the optimization problem with very general constraints over time-varying unbalanced directed topologies.

The remaining parts are organized as follows. Section~\ref{sec:prel} recalls some preliminaries, including notations, graph theory, and useful lemmas.
Section~\ref{sec:pam} shows the problem formulation, presents the algorithm development, and gives the main theorems on the convergence property and the convergence rate.
In Section~\ref{sec:PT}, the detailed proofs of the main theorems are made.
Section~\ref{sec:simu} provides some numerical simulations and Section~\ref{sec:conc} gives a summary of the work finally.
Besides, the proofs of intermediate results are furnished in the Appendix.

\section{Preliminaries}\label{sec:prel}


\subsection{Notations}\label{subsec:notations}

Let $\mathbb{R}^{n}$ denote the $n$-dimensional real-valued vectors set and $\mathbf{1}$ denote the vector with proper dimension and all its entries being $1$. Let $\mathbb{R}^{N\times N}$ represent the set of $N$-dimensional real-valued square matrices and $I_{N}$ represent the $N$-dimensional identity matrix. Furthermore, for a given matrix $M\in \mathbb{R}^{N\times N}$ and $i,j=1,2,\cdots,N$, $M_{ij}$ or $[M]_{ij}$ denotes the $ij$-th entry of matrix $M$. Particularly, for a given vector $z\in \mathbb{R}^{n}$, $z^{T}$ is the transpose of $z$, $\|z\|$ stands for the Euclidean norm of $z$, and the scalar $z^{i}$ represents the $i$-th entry of $z$. Moreover, for a given convex function $f$$:\mathbb{R}^{n}\rightarrow \mathbb{R}$, $\nabla f(x)\ (\partial f(x))$ stands for the gradient (sub-gradient) of function $f$ in $x$. Besides, for a scalar $a$, $a^{+}=\max\{a,0\}$. Additionally, for a given set $\Omega\subseteq \mathbb{R}^{n}$, $\textbf{relint}\,\Omega$ denotes the relative interior of the set $\Omega$.

\subsection{Graph Theory}\label{subsec:graph}

Let $\{\mathcal{G}(t)\}$ stand for a directed graph sequence with a common node set $\mathcal{N}=\{1,2,\cdots,N\}$ and a directed edge set sequence $\{\mathcal{E}(t)\subseteq\mathcal{N}\times\mathcal{N}$\}, where $t=0,1,2,\cdots$. Further, for a given natural number $t$, $(i,j)\in\mathcal{E}(t)$ means that there exists a directed edge from node $j$ to node $i$, that is, node $i$ can receive information from node $j$, node $j$ is an in-neighbor of node $i$ and meanwhile node $i$ is an out-neighbor of node $j$. Moreover, define $\mathcal{N}_{i}^{in}(t)$ as the set of all node $i$'s in-neighbors and $d_{i}^{+}(t)$ as its cardinality at $t$. Similarly, define $\mathcal{N}_{i}^{out}(t)$ as the collection of all node $i$'s out-neighbors and $d_{i}^{-}(t)$ as its cardinality at $t$. Particularly, we always let $i\in\mathcal{N}_{i}^{in}(t)\bigcap\mathcal{N}_{i}^{out}(t)$, that is, $(i,i)\in\mathcal{E}(t)$, for all $t\geq0$. For a given $t$, a directed path from node $i_{m}$ and node $i_{1}$ is defined as a sequence of $m$ distinct nodes $i_{1},\cdots,i_{m}$ such that $(i_{q},i_{q+1})\in\mathcal{E}(t)$ with $q=1,2,\cdots,m-1$. A static directed graph $\mathcal{G}$ is called a strongly connected graph if there exists a directed path between any two distinct nodes in graph $\mathcal{G}$. For a given natural number $t$ and a positive integer $H$, the joint graph $\bigcup\limits_{i=t}^{t+H-1}\mathcal{G}(i)$ denotes the graph with the node set $\mathcal{N}=\{1,2,\cdots,N\}$ and a directed edge set consisting of the union of the directed edge sets of graphs $\mathcal{G}(t),\mathcal{G}(t+1),\cdots,\mathcal{G}(t+H-1)$. Furthermore, the directed graph sequence $\{\mathcal{G}(t)\}$ is said to be uniformly jointly strongly connected if for all $t\geq0$, there is a positive integer $H$ such that the joint graph $\bigcup\limits_{i=t}^{t+H-1}\mathcal{G}(i)$
is strongly connected.


%

%

\subsection{Useful Lemmas}\label{subsec:lemma}

In this subsection, some auxiliary results are shown for facilitating analysis of the main results.

\begin{lemma}[\hspace{-0.01em}\cite{Nedic2010TAC,Mai2019Aut}]\label{lemma1}
Suppose that $Y\subseteq \mathbb{R}^{n}$ is a closed convex set and $P_{Y}$ is the projection operator on $Y$, i.e., $P_{Y}(u)=\argmin_{v\in Y}\|u-v\|$. Then, for all $x_{1}\ ,x_{2}\in \mathbb{R}^{n}$, and $x_{3}\in Y$, there hold
\begin{enumerate}[(a)]
\item $\|P_{Y}(x_{1})-P_{Y}(x_{2})\|\leq\|x_{1}-x_{2}\|$;
\item $\|P_{Y}(x_{1})-x_{3}\|^{2}\leq\|x_{1}-x_{3}\|^{2}-\|P_{Y}(x_{1})-x_{1}\|^{2}$.
\end{enumerate}
\end{lemma}


\begin{lemma}[\hspace{-0.02cm}\cite{Nedic2015TAC}]\label{lemma2}
\begin{enumerate}[(a)]
\item For a scalar sequence $\{\theta(k)\}$, suppose that $\lim\limits_{k\to\infty}\theta(k)=\theta$ and $0<\rho<1$ hold. Then $\lim\limits_{k\to\infty}\sum\limits_{l=0}^{k}\rho^{k-l}\theta(l)=\frac{\theta}{1-\rho}$.
\item For a strictly positive scalar sequence $\{\theta(k)\}$, assume that $\sum\limits_{k=0}^{\infty}\theta(k)<\infty$ and $0<\rho<1$. Then $\sum\limits_{k=0}^{\infty}\bigg(\sum\limits_{l=0}^{k}\rho^{k-l}\theta(l)\bigg)<\infty$.
\end{enumerate}
\end{lemma}

\begin{lemma}[\hspace{-0.02cm}\cite{Nedic2015TAC}]\label{lemma3}
For nonnegative scalar sequences $\{a(k)\}$, $\{b(k)\}$, $\{c(k)\}$ and $\{d(k)\}$ with $\sum\limits_{k=0}^{\infty}b(k)<\infty$ and $\sum\limits_{k=0}^{\infty}c(k)<\infty$, let the following condition hold:
\begin{equation*}
a(k+1)\leq(1+b(k))a(k)-d(k)+c(k),\ \forall k\geq0.
\end{equation*}
Then, the scalar sequence $\{a(k)\}$ converges to some $a\geq0$ and $\sum\limits_{k=0}^{\infty}d(k)<\infty$.
\end{lemma}

In the next Lemmas~\ref{lemma4} and~\ref{lemma5}, we reveal some important properties of the row stochastic matrices sequence and the column stochastic matrices sequence associated with a sequence of uniformly jointly strongly connected directed graphs, which are the key points for the subsequent analysis of consensus and our improved push-pull mechanism.

\begin{lemma}[\hspace{-0.02cm}\cite{Li2019TAC}]\label{lemma4}
For a uniformly jointly strongly connected directed graphs sequence $\{\mathcal{G}(k)\}$ and $k\geq s\geq0$, let $A(k:s)=A(k)\cdots A(s)$ when $k>s$, and $A(k:s)=A(k)$ when $k=s$, with $\{A(k)\}$ being a sequence of nonnegative row stochastic matrices associated with $\{\mathcal{G}(k)\}$. Also, assume that there exists a positive constant $a_{0}$ such that $A_{ii}(k)\geq a_{0}$ hold for all $i=1,2,\cdots,N$ and $k\geq0$. Then, $\forall k\geq0$, there exists a positive vector sequence $\{\pi(k)\}$ satisfying $\pi^{T}(k)\mathbf{1}=1$ and the following conditions:
\begin{enumerate}[(a)]
  \item For all $i,j\in\mathcal{N}$ and $k\geq s$, there are two constants $C_{1}>0$ and $0<\lambda_{1}<1$ such that $|a_{ij}(k:s)-\pi^{j}(s)|\leq C_{1}\lambda_{1}^{k-s}$;
  \item For all $i\in\mathcal{N}$ and $k\geq0$, there is a strictly positive constant $\theta_{1}$ such that $\pi^{i}(k)\geq\theta_{1}$;
  \item For all $k\geq0$, $\pi^{T}(k)=\pi^{T}(k+1)A(k)$.
\end{enumerate}
\end{lemma}

\begin{lemma}[\hspace{-0.02cm}\cite{Nedic2015TAC}]\label{lemma5}
For a uniformly jointly strongly connected directed graphs sequence $\{\mathcal{G}(k)\}$ and $k\geq s\geq0$, let $B(k:s)=B(k)\cdots B(s)$ when $k>s$, and $B(k:s)=B(k)$ when $k=s$, with $\{B(k)\}$ being a sequence of nonnegative column stochastic matrices associated with $\{\mathcal{G}(k)\}$. Furthermore, assume that there exists a positive constant $b_{0}$ such that $B_{ii}(k)\geq b_{0}$ hold for all $i=1,2,\cdots,N$ and $k\geq0$. Then, $\forall k\geq0$, there exists a positive vector sequence $\{\mu(k)\}$ satisfying $\mu^{T}(k)\mathbf{1}=1$ and the following conditions:
\begin{enumerate}[(a)]
  \item For all $i,j\in\mathcal{N}$ and $k\geq s$, there are two constants $C_{2}>0$ and $0<\lambda_{2}<1$ such that $|b_{ij}(k:s)-\mu^{i}(k)|\leq C_{2}\lambda_{2}^{k-s}$;
  \item For all $i\in\mathcal{N}$ and $k\geq0$, there is a strictly positive constant $\theta_{2}$ such that $\mu^{i}(k)\geq\theta_{2}$.
\end{enumerate}
\end{lemma}

Moreover, in the following Lemma~\ref{lemma6}, we provide a key result for finding the point $x$ satisfying $g(x)\leq0$, with $g(x)$ being a convex function.

\begin{lemma}[\hspace{-0.02cm}\cite{Polyak1969}]\label{lemma6}
Let $Y\subseteq \mathbb{R}^{n}$ be an arbitrary closed convex set, $g:\mathbb{R}^{n}\rightarrow \mathbb{R}$ be a convex function, and $x\in \mathbb{R}^{n}$ be given by
\begin{equation}\label{lemma6_eq1}
x=P_{Y}\bigg(v-\beta_{0}\frac{g^{+}(v)}{\|d\|^{2}}d\bigg),
\end{equation}
where $v\in \mathbb{R}^{n}$, $0<\beta_{0}<2$ is a constant, and $d\in\partial g^{+}(v)$ when $g^{+}(v)>0$ and $d=d_{0}\neq0$ otherwise, with $d_{0}$ being an arbitrary nonzero constant vector. Then, for all $z\in Y$ satisfying $g^{+}(z)=0$, there holds
\begin{equation}\label{lemma6_eq2}
\|x-z\|^{2}\leq\|v-z\|^{2}-\beta_{0}(2-\beta_{0})\frac{(g^{+}(v))^{2}}{\|d\|^{2}}.
\end{equation}
\end{lemma}

Based on (\ref{lemma6_eq1}) given in Lemma~\ref{lemma6}, in order to find the point $x$ satisfying $g(x)\leq0$ with $g(x)$ being a convex function, an iterative rule can be correspondingly designed as
\begin{subequations}\label{pdfi}
\begin{align}\label{pdfi-a}
x(t+1)=P_{Y}\bigg(x(t)-\beta_{0}\frac{g^{+}(x(t))}{\|d^{\prime}\|^{2}}d^{\prime}\bigg),
\end{align}
with an arbitrary initial value $x(0)$, and $d^{\prime}\in\partial g^{+}(x(t))$ when $g^{+}(x(t))>0$ and $d^{\prime}=d_{0}\neq0$ otherwise. Clearly, the above iterative rule (\ref{pdfi-a}) can be equivalently written as
\begin{align}\label{pdfi-b}
x(t+1)=P_{Y}\big(x(t)-\beta(x(t))\partial g^{+}(x(t)))\big),
\end{align}
\end{subequations}
which can be approximately seen as the classical projected subgradient decent iteration, with the nonnegative step-size $\beta(x(t))=\beta_{0}\frac{g^{+}(x(t))}{\|d\|^{2}}$ if $g^{+}(x(t))>0$ and $\beta(x(t))=0$ otherwise. Since $\beta(x(t))$ is positive provided that $g^{+}(x(t))>0$, the value $g^{+}(x(t+1))$ will become smaller than $g^{+}(x(t))$ when $g^{+}(x(t))>0$, which means that the distance from $x(t+1)$ to the set $\{y\mid g(y)\leq0\}$ will be smaller than the distance from $x(t)$ to the set $\{y\mid g(y)\leq0\}$. Furthermore, for an arbitrary $z\in\{y\mid g(y)\leq0\}$, the relation between $x(t+1)$ and $z$ can be characterized as
\begin{align*}
\|x(t+1)-z\|^{2}\leq\|x(t)-z\|^{2}-\beta_{0}(2-\beta_{0})\frac{(g^{+}(x(t)))^{2}}{\|d^{\prime}\|^{2}}.
\end{align*}
In addition, it can also be concluded that the above iteration rule (\ref{pdfi-b}) will stop at the point $x$ satisfying $g^{+}(x(t))=0$ or $\partial g^{+}(x(t))=0$, both of which can imply $g(x)\leq0$ with the consideration that $g(x)$ is convex.


\section{Main Results}\label{sec:pam}

\subsection{Problem Formulation}\label{sec:pf}

In this paper, the optimization problem with $N$ nonidentical inequality constraints and $N$ nonidentical closed convex set constraints is researched, which is formulated as
\begin{equation}\label{problem_eq1}
\begin{split}
    \min &~ f(x)= \sum\limits_{i=1}^{N}f_{i}(x) \\
    {\rm s.t.} &~~ g_{i}(x)\leq0,\\
               &~~ x\in X_i,\ i=1,2,\cdots,N,
  \end{split}
\end{equation}
where $x\in \mathbb{R}^{n}$, $f_{i}:\mathbb{R}^{n}\rightarrow\mathbb{R}$ and $g_{i}:\mathbb{R}^{n}\rightarrow\mathbb{R}$ are convex functions, and $X_{i}\subseteq \mathbb{R}^{n}$ are closed convex sets. Furthermore, a directed graphs sequence $\{\mathcal{G}(k)\}$ is introduced to describe local interactions between nodes in this paper. 
Slater's condition is first assumed for the problem (\ref{problem_eq1}).
\begin{assumption}\label{asm3}
There exists a vector $\hat{x}\in\textbf{relint}\,X=\bigcap\limits_{i=1}^{N} X_{i}$ such that $g_{i}(\hat{x})<0$ for all $i=1,2,\cdots,N$.
\end{assumption}

The optimality conditions for the problem setup are as follows.
Let $X_{0}$ and $X^{*}$ denote the feasible solutions set and the optimal solutions set of the considered problem, respectively. Specially, under Assumption~\ref{asm3}, we can get that $X_{0}$ and $X^{*}$ are nonempty. Meanwhile, let $x^{*}$ be an optimal solution to the problem (\ref{problem_eq1}) and assume that the optimal function value $f^{*}$ is finite. Particularly, the following properties of $x^{*}$ play an important role in the subsequent analysis:
\begin{equation*}
f(x^{*})\leq f(x) \text{ for all }x\in X_{0},\ g_{0}(x^{*})\leq0,\text{ and }x^{*}\in X,
\end{equation*}
with $g_{0}(x)=\max\{g_{1}(x),g_{2}(x),\cdots,g_{N}(x)\}$.

\subsection{Algorithm Development}\label{sec:ad}

First, to solve the above optimization problem (\ref{problem_eq1}), a distributed discrete-time algorithm is developed as below:
\begingroup
\allowdisplaybreaks
\begin{subequations}\label{algorithm1_eq1}
\begin{align}
v_{i}(t)&=\sum\limits_{j=1}^{N}A_{ij}(t)x_{j}(t)-y_{i}(t),\label{algorithm1_eq1a}\\
x_{i}(t+1)&=P_{X_{i}}\big(v_{i}(t)-\beta k_{i}(t)\big),\label{algorithm1_eq1b}\\
y_{i}(t+1)&=\sum\limits_{j=1}^{N}B_{ij}(t)y_{j}(t)+\alpha(t+1)\nabla f_{i}(x_{i}(t+1))-\alpha(t)\nabla f_{i}(x_{i}(t)),\label{algorithm1_eq1c}
\end{align}
\end{subequations}
\endgroup
where
\vspace{-0.5ex}
\begin{equation*}
k_{i}(t)=\frac{g^{+}_{i}(v_{i}(t))}{\|d_{i}(t)\|^{2}}d_{i}(t),
\end{equation*}
with
\vspace{-0.5ex}
\begin{equation*}
d_{i}(t)\!=\!\left\{\!\!\!
\begin{array}{l@{\hspace{0.8em}}l}
\partial g_{i}^{+}(v_{i}(t)),
           & g_{i}^{+}(v_{i}(t))\neq0,\\
d_{0} \mbox{ (a nonzero constant vertor)}, & g_{i}^{+}(v_{i}(t))=0.
\end{array}\right.
\end{equation*}
Moreover, $\{A(t)\}\subseteq \mathbb{R}^{N\times N}$ is a nonnegative row stochastic weight matrices sequence associated with the given directed graphs sequence $\{\mathcal{G}(t)\}$, whose entries can be defined as $A_{ij}(t)=\frac{1}{d_{i}^{+}(t)}$ if $(i,j)\subseteq\mathcal{E}(t)$ and $A_{ij}(t)=0$ otherwise; and $\{B(t)\}\subseteq \mathbb{R}^{N\times N}$ is a nonnegative column stochastic weight matrices sequence associated with the given directed graphs sequence $\{\mathcal{G}(t)\}$, whose entries can be defined as $B_{ij}(t)=\frac{1}{d_{j}^{-}(t)}$ if $(i,j)\subseteq\mathcal{E}(t)$ and $B_{ij}(t)=0$ otherwise.
In addition, $\{\alpha(t)\}$ is a positive decaying step-size sequence satisfying $\sum\limits_{t=0}^{\infty}\alpha(t)=\infty$ and $\sum\limits_{t=0}^{\infty}\alpha^{2}(t)<\infty$, and $\beta$ is a constant step-size satisfying $0<\beta<2$. Specially, we select $x_{i}(0)$ as an arbitrary value and $y_{i}(0)=\alpha(0)\nabla f_{i}(x_{i}(0))$.

It is worth mentioning that the algorithm~(\ref{algorithm1_eq1}) is designed on the basis of the push-pull mechanism with some significant improvements. Thus, we first give some necessary description of the push-pull based algorithms before providing some insights into the algorithm~(\ref{algorithm1_eq1}). In \cite{Pu2018CDC}, the push-pull/AB algorithm was developed as
\begingroup
\allowdisplaybreaks
\begin{align*}
x_{i}(t+1)&=\sum\limits_{j=1}^{N}A_{ij}x_{j}(t)-\alpha y_{i}(t),\\
y_{i}(t+1)&=\sum\limits_{j=1}^{N}B_{ij}y_{j}(t)+\nabla f_{i}(x_{i}(t+1))-\nabla f_{i}(x_{i}(t)),
\end{align*}
\endgroup

\vspace{-1ex}
\noindent
where $\{A_{ij}\}$ and $\{B_{ij}\}$ are respectively a nonnegative row stochastic matrices sequence and a nonnegative column stochastic matrices sequence associated with the given static unbalanced graph, and $\alpha$ is a desired constant step-size. Moveover, the initial values $x_{i}(0)$ can be arbitrarily selected and $y_{i}(0)=\nabla f_{i}(x_{i}(0))$ for all $i$. Different kinds of push-pull/AB algorithms were designed for different cases in \cite{Xin2015SCL,Xin2020TAC,Saadatniaki2018} and the more detailed information on the push-pull/AB algorithms can be found in \cite{Xin2020Proc}. Clearly, it can be seen from \cite{Xin2015SCL,Pu2018CDC,Xin2020TAC,Saadatniaki2018} that the constant step-sizes were usually employed in the iterations of the state variables $x_{i}(t)$ in the push-pull/AB algorithms. Furthermore, it can also be noticed from the above existing results that the convergence analysis of the push-pull/AB algorithms depends heavily on constructing some linear matrix inequalities with the strong convexity properties being imposed on all local objective functions and only the unconstrained optimization problems can be dealt with. Nowadays, it is still rather difficult to solve the constrained optimization problem by employing the push-pull/AB algorithms, which remains as an open issue.

In this work, in order to effectively integrate the methods of handling the nonidentical local inequality constraints and closed convex set constraints into the classical push-pull mechanism, the decaying step-size $\alpha(t)$ is newly employed in our improved push-pull mechanism. Specifically, the decaying step-size $\alpha(t)$ is involved in the iterations of the auxiliary variables $y_{i}(t)$ rather than the state variables $x_{i}(t)$ (see the iterative rule (\ref{algorithm1_eq1c})). Thus, the auxiliary variables $y_{i}(t)$ can track the gradient terms $\alpha(t)\mu^{i}(t-1)\sum\limits_{i=1}^{N}\nabla f_{i}(x_{i}(t))$ with $\mu^{i}(t)$ as introduced in Lemma~\ref{lemma5}, and the tracking errors can be clearly characterized by some nonnegative variables which can be well analyzed under our improved push-pull mechanism now. Accordingly, the iterative rule (\ref{algorithm1_eq1b}) can be approximately written as
\begin{equation}\label{eqn:approx_al}
x_{i}(t+1)=P_{X_{i}}(v_{i}^{\prime}(t)-\beta k^{\prime}_{i}(t)),
\end{equation}
where $v_{i}^{\prime}(t)$ is defined by replacing all terms $y_{i}(t)$ in $v_{i}(t)$ with $\alpha(t)\mu^{i}(t-1)\sum\limits_{i=1}^{N}\nabla f_{i}(x_{i}(t))$ and $k^{\prime}_{i}(t)$ is defined by replacing all terms $v_{i}(t)$ in $k_{i}(t)$ with $v_{i}^{\prime}(t)$. Based on Lemma~\ref{lemma6}, it can be obtained that $v_{i}^{\prime}(t)$ will move to the set $\{y\mid g_{i}(y)\leq0\}$. Moreover, noting that $\lim\limits_{t\to\infty}\alpha(t)=0$ and $\mu^{i}(t-1)\sum\limits_{i=1}^{N}\nabla f_{i}(x_{i}(t))$ is uniformly bounded with respect to $t$, we can get that $z_{i}(t)=\sum\limits_{j=1}^{N}A_{ij}(t)x_{j}(t)$ will also move to the set $\{y\mid g_{i}(y)\leq0\}$. Additionally, since $\{A(t)\}$ is the nonnegative row stochastic matrices sequence, the term $\sum\limits_{j=1}^{N}A_{ij}(t)x_{j}(t)$ will enable $x_{i}(t)$ to move towards the varying consensus state $\sum\limits_{i=1}^{N}\pi^{i}(t)x_{i}(t)$ with $\pi^{i}(t)$ as introduced in Lemma~\ref{lemma4}. Thus, it follows that
\begin{equation*}
\lim\limits_{t\to\infty}\bigg\|x_{i}(t)-\sum\limits_{i=1}^{N}\pi^{i}(t)x_{i}(t)\bigg\|=\lim\limits_{t\to\infty}\|x_{i}(t)-z_{i}(t)\|=0,
\end{equation*}
i.e., $x_{i}(t)$ will also move to the set $\{y\mid g_{i}(y)\leq0\}$ and then to the set $\{y\mid g_{i}(y)\leq0,i=1,2,\cdots,N\}$ since all $x_{i}(t)$ will reach consensus. Furthermore, each varying state $x_{i}(t)$ is also controlled by the projection operation on the set $X_{i}$ and pushed by the gradient term $\alpha(t)\mu^{i}(t-1)\sum\limits_{i=1}^{N}\nabla f_{i}(x_{i}(t))$ to move towards the set $X_{i}$ and then to the optimal solution set of the optimization problem
\begin{equation}\label{problem3}
    \min\, \sum\limits_{i=1}^{N} f_{i}(x)\ {\rm s.t.}\ g_{i}(x)\leq0\,\ x\in X_i,\ i=1,2,\cdots,N.
\end{equation}
Consequently, upon reaching consensus, all states will converge to a common optimal solution of the considered problem (\ref{problem_eq1}).

Here, we also provide some insights into the algorithm (\ref{algorithm1_eq1}) from the perspective of centralized iteration. The corresponding centralized iteration of the algorithm~(\ref{algorithm1_eq1}) can be formulated as
\begin{equation}\label{eqn:ci}
x(t+1)=P_{X}(x(t)-\alpha(t)\nabla f(x(t))-\beta k(t)),
\end{equation}
where
\begin{equation*}
k(t)=\left\{
\begin{array}{l@{\hspace{1em}}l}
\displaystyle
\frac{g^{+}_{0}(v(t))}{\|\partial g_{0}^{+}(v(t))\|^{2}}\partial g_{0}^{+}(v(t)),
           & g_{0}^{+}(v(t))\neq0,\vspace{0.5ex}\\
0, & \mbox{otherwise},
\end{array}\right.
\end{equation*}
with $v(t)=x(t)-\alpha(t)\nabla f(x(t))$.
Moreover, the convergence property of the above centralized iteration (\ref{eqn:ci}) can be deduced from Lemma~\ref{lemma3} and Lemma~\ref{lemma6}, which can be found in \cite{Nedic2011MP}. Indeed, in the remainder of the paper, we will resort to the techniques used in the convergence analysis of the centralized iteration together with characterizing the error terms $\|y_{i}(t)-\alpha(t)\mu^{i}(t-1)\sum\limits_{i=1}^{N}\nabla f_{i}(x_{i}(t))\|$ and $\alpha(t)\|x_{i}(t)-\sum\limits_{i=1}^{N}\pi^{i}(t)x_{i}(t)\|$ by some well analyzed nonnegative variables to complete the convergence analysis of the developed distributed algorithm (\ref{algorithm1_eq1}), which will be clearly shown in Section~\ref{sec:PT}.

\begin{remark}\label{remark1}
It can be concluded from the above discussion that some significant improvements on the classical push-pull/AB algorithms are made in our paper. First, the decaying step-size $\alpha(t)$ is newly employed in our improved push-pull mechanism and the decaying step-size is involved in the iterations of the auxiliary variables $y_{i}(t)$ rather than the state variables $x_{i}(t)$. Second, the methods of handling the nonidentical local inequality constraints and closed convex set constraints are effectively integrated into our improved push-pull mechanism. As a consequence, the convergence analysis of the algorithm~(\ref{algorithm1_eq1}), as to be shown in Section~\ref{sec:PT}, is conducted based on Lemma~\ref{lemma3}, which is absolutely distinct from the convergence analyses as given in \cite{Xin2015SCL,Pu2018CDC,Xin2020TAC,Saadatniaki2018} for the existing push-pull/AB algorithms. Third, it is not necessary to assume the strong convexity of the local objective functions or to construct any linear matrix inequalities for proceeding with the convergence analysis. More importantly, the optimization problem with multiple nonidentical local constraints is successfully addressed over a sequence of time-varying unbalanced graphs under our improved push-pull mechanism in this work, which is known as a very challenging issue in the field of distributed optimization so far. Furthermore, it will also be convenient to extend the results obtained in this paper to the case with other kinds of distributed constrained optimization problems.
\end{remark}

\subsection{Main Theorems}\label{sec:mt}

Before going on, the definition of $L$-smooth function and several common assumptions are introduced for the establishment of the main theorems.

\begin{definition}\label{def1}
A differentiable function $h$ : $\mathbb{R}^{n}\rightarrow \mathbb{R}$ is said to be $L$-smooth if the following condition holds: $\forall x_{1}, x_{2}\in \mathbb{R}^{n}$,
\begin{equation*}
\|\nabla h(x_{1})-\nabla h(x_{2})\|\leq L\|x_{1}-x_{2}\|.
\end{equation*}
\end{definition}

\begin{assumption}\label{asm1}
For all $i=1,2,\cdots,N$, functions $f_{i}(x)$ are $L$-smooth and functions $g_{i}(x)$ have continuous gradients. 
\end{assumption}

\begin{assumption}\label{asm2}
For all $i=1,2,\cdots,N$, $X_{i}\subseteq \mathbb{R}^{n}$ are compact sets.
\end{assumption}

\begin{proposition}\label{prop2}
There exists a constant $R$ such that
\begin{equation*}
\text{dist}(x,X_{0})\leq R\max\Big\{\max\limits_{1\leq i\leq N}\text{dist}(x,X_{i}),\max\limits_{1\leq i\leq N}g_{i}^{+}(x)\Big\},
\end{equation*}
for all $x\in U=\text{conv}\Big(\bigcup\limits_{i=1}^{N}X_{i}\Big)$.
\end{proposition}

Clearly, under Assumptions~\ref{asm2}, Proposition~\ref{prop2} will naturally hold since the involved variables are uniformly bounded when $x\in U$. Moreover, Proposition~\ref{prop2} (also known as the constraint regularity property) is very important and standard in dealing with the optimization problems involving nonidentical constraints, which can also be found in \cite{Mai2019Aut} and \cite{Nedic2011MP}.

\begin{assumption}\label{asm4}
The time-varying directed graph sequence $\{\mathcal{G}(t)\}$ is uniformly jointly strongly connected.
\end{assumption}

\begin{remark}\label{remark3}
Under Assumptions~\ref{asm1} and~\ref{asm2}, we can conclude that $\nabla f_{i}(x_{i}(t))$ can be uniformly bounded by a positive constant $M$, i.e., $\|\nabla f_{i}(x_{i}(t))\|\leq M$, since $\nabla f_{i}(x_{i}(t))$ are continuous and $x_{i}(t)$ generated by the iteration rule (\ref{algorithm1_eq1b}) are contained in the compact set $U$. Besides, Assumption~\ref{asm4} is also very common and standard in dealing with the optimization problems over time-varying unbalanced directed graphs.
\end{remark}

Now, we are prepared to develop our first main theorem, which demonstrates the convergence property of the developed distributed discrete-time algorithm (\ref{algorithm1_eq1}).

\begin{theorem}\label{theorem1}
Let Assumptions~\ref{asm3}, \ref{asm1}, \ref{asm2}, and~\ref{asm4} hold. Then for all $i=1,2,\cdots,N$, the states $x_{i}(t)$ generated by the algorithm (\ref{algorithm1_eq1}) converge to a common optimal solution $s^{*}$ of the problem (\ref{problem_eq1}), i.e., $\lim\limits_{t\to\infty}x_{i}(t)=s^{*}$.
\end{theorem}

\begin{proof}
The detailed proof will be given in the subsequent Subsection~\ref{sec:PT-Thm1}.
\end{proof}

Next, defining the following two variables:
\begingroup
\allowdisplaybreaks
\begin{align*}
\tilde{x}_{i}(t)=\frac{\sum\limits_{k=0}^{t}\alpha(k)x_{i}(k)}{\sum\limits_{k=0}^{t}\alpha(k)},\ \ \tilde{s}(t)=\frac{\sum\limits_{k=0}^{t}\alpha(k)s(k)}{\sum\limits_{k=0}^{t}\alpha(k)},
\end{align*}
\endgroup
with $s(t)=P_{X_{0}}(\bar{x}(t))$ and $\bar{x}(t)=\sum\limits_{i=1}^{N}\pi^{i}(t)x_{i}(t)$, we state our second main theorem showing the rate at which the function values $f(\tilde{x}_{i}(t))$ converge to the optimal value $f^{*}$.

\begin{theorem}\label{theorem2}
Let Assumptions~\ref{asm3}, \ref{asm1}, \ref{asm2}, and~\ref{asm4} hold. Then for the algorithm (\ref{algorithm1_eq1}), there holds
\begin{equation}\label{them2_eq1}
H_{1}\|\tilde{x}_{i}(t)-\tilde{s}(t)\|+f(\tilde{s}(t))-f^{*}\leq L(t),
\end{equation}
with
\begingroup
\allowdisplaybreaks
\begin{align}\label{them2_eq2}
L(t)=\bigg(H_{2}+H_{3}\sum\limits_{k=0}^{t}\alpha^{2}(k)\bigg)\bigg/\sum\limits_{k=0}^{t}\alpha(k),
\end{align}
\endgroup
and $H_{j}$, $j=1,2,3$, being positive constants which will be defined in the subsequent proof of Theorem~\ref{theorem2}. Furthermore, for each $i=1,2,\cdots,N$, there holds
\begingroup
\allowdisplaybreaks
\begin{align}\label{them2_eq3}
|f(\tilde{x}_{i}(t))-f^{*}|\leq\bigg(\frac{NM}{H_{1}}+1\bigg)L(t),
\end{align}
\endgroup
with $M$ being the positive constant as given in Remark~\ref{remark3}.
\end{theorem}

\begin{proof}
The detailed proof will be given in the subsequent Subsection~\ref{sec:PT-Thm2}.
\end{proof}

\begin{remark}
Clearly, we can select $\mathcal{O}(\frac{1}{t+1})^{\sigma}$ with $0.5<\sigma\leq1$ as the decaying step-size $\alpha(t)$. Now, let $\alpha(t)=\mathcal{O}\Big(\frac{1}{\sqrt{t+1}}\Big).$
By direct computation, it is easy to get
\begin{align*}
L(t)\sim\mathcal{O}\bigg(\frac{\ln t}{\sqrt{t}}\bigg),
\end{align*}
which means that the terms $\|\tilde{x}_{i}(t)-\tilde{s}(t)\|$, $f(\tilde{s}(t))-f^{*}$ and $|f(\tilde{x}_{i}(t))-f^{*}|$ will all decay to zero at a rate of $\mathcal{O}\big(\frac{\ln t}{\sqrt{t}}\big)$.
This result for the convergence rate is very common and can also be found in the recent works \cite{Nedic2015TAC} and \cite{Mai2019Aut}.
\end{remark}

\section{Proofs of Main Theorems}\label{sec:PT}

\subsection{Road Map for Proofs of Main Theorems}\label{sec:PT-Map}

To begin with, we briefly sketch the main ideas of the proofs of Theorems~\ref{theorem1} and \ref{theorem2}.
\begin{enumerate}[1)]
\item First, we will start the analysis in Lemmas~\ref{lemma7} and~\ref{lemma8} by discussing that under Assumptions~\ref{asm1} and~\ref{asm2}, the auxiliary variables $y_{i}(t)$ can track the global gradient terms $\alpha(t)\mu^{i}(t-1)\sum\limits_{i=1}^{N}\nabla f_{i}(x_{i}(t))$.

\item Second, to provide some preparations for the subsequent analysis, we will argue in Remark~\ref{remark6} that $k_{i}(t)$ in the proposed distributed discrete-time algorithm (\ref{algorithm1_eq1}) can be uniformly bounded with respect to $t$ under Assumption~\ref{asm3}.

\item Third, we will preliminarily describe the consensus errors in Lemma~\ref{lemma9}, based on which we will further characterize the gradient error terms $\|y_{i}(t)-\alpha(t)\mu^{i}(t-1)\sum\limits_{i=1}^{N}\nabla f_{i}(x_{i}(t))\|$ in Lemma~\ref{lemma10}.

\item Fourth, we will analyze the evolution of all states $x_{i}(t)$ in Lemmas~\ref{lemma11} and~\ref{lemma12} and show that, as a consequence of all states reaching consensus, the iteration (\ref{algorithm1_eq1b}) can be approximately seen as the ordinary iteration for solving the problem (\ref{problem3}) and all states $x_{i}(t)$ will converge to a common optimal solution of the considered problem (\ref{problem_eq1}).

\item Finally, with all these preparatory works, we will be ready to give the detailed proofs of our main Theorems~\ref{theorem1} and~\ref{theorem2}.
\end{enumerate}

In addition, all the proofs of several aforementioned intermediate results (including Lemmas~\ref{lemma7}--\ref{lemma12} and Corollary~\ref{cor1}) will be collected in the Appendix so as to maintain a smooth presentation flow.

\subsection{Preparatory Works}\label{sec:PT-Pre}

First, we define the following two intermediate variables:
\begin{equation*}
\tilde{y}(t)=\sum\limits_{i=1}^{N}y_{i}(t),\quad \tilde{f}(t)=\sum\limits_{i=1}^{N}f_{i}(x_{i}(t)).
\end{equation*}
Motivated by the results in \cite{Nedic2015TAC}, \cite{Qu2017TCNS} and \cite{Xi2018TAC}, we will preliminarily characterize the error between the auxiliary variable $y_{i}(t)$ and the global gradient term $\alpha(t)\mu^{i}(t-1)\sum\limits_{i=1}^{N}\nabla f_{i}(x_{i}(t))$ in the following Lemmas~\ref{lemma7} and~\ref{lemma8}.

\begin{lemma}\label{lemma7}
For the algorithm (\ref{algorithm1_eq1}) with given initial values and all $t\geq0$, there holds
\begingroup
\allowdisplaybreaks
\begin{align*}
\tilde{y}(t)=\alpha(t)\nabla\tilde{f}(t).
\end{align*}
\endgroup
\end{lemma}



Note that with Lemma~\ref{lemma7}, we have
\begin{align*}
\alpha(t)\mu^{i}(t-1)\sum\limits_{i=1}^{N}\nabla f_{i}(x_{i}(t))
=\mu^{i}(t-1)\tilde{y}(t).
\end{align*}

\begin{lemma}\label{lemma8}
Let Assumptions~\ref{asm3}, \ref{asm1}, \ref{asm2}, and~\ref{asm4} hold. Then, for the algorithm (\ref{algorithm1_eq1}), the following inequality holds:
\begingroup
\allowdisplaybreaks
\begin{align}\label{lemma8_eq1}
\|y_{i}(t)-\mu^{i}(t-1)\tilde{y}(t)\|
\leq C_{3}\lambda_{2}^{t-1}+C_{4}\sum\limits_{s=0}^{t-1}\lambda_{2}^{t-1-s}\alpha(s),
\end{align}
\endgroup
where $C_{3}=C_{2}\sum\limits_{j=1}^{N}\|y_{j}(0)\|$ and $C_{4}=2NMC^{\prime}_{4}$ with $C^{\prime}_{4}=\max\{\frac{C_{2}}{\lambda_{2}},1\}$.
\end{lemma}

Then, based on Lemmas~\ref{lemma7} and~\ref{lemma8}, we give the following corollary where a bound is introduced for the variable $y_{i}(t)$.

\begin{corollary}\label{cor1}
Let Assumptions~\ref{asm3}, \ref{asm1}, \ref{asm2}, and~\ref{asm4} hold. Then, for the algorithm (\ref{algorithm1_eq1}), the following inequality holds:
\begin{align}\label{cor_eq1}
\|y_{i}(t)\|\leq C_{3}\lambda_{2}^{t-1}+C_{4}\sum\limits_{s=0}^{t-1}\lambda_{2}^{t-1-s}\alpha(s)+NM\alpha(t).
\end{align}
Furthermore, there holds $\lim\limits_{t\to\infty}\|y_{i}(t)\|=0$.
\end{corollary}

\begin{remark}\label{remark6}
From the result in Corollary~\ref{cor1} and the iteration rule (\ref{algorithm1_eq1a}), it is clearly seen that $v_{i}(t)$ is uniformly bounded with respect to $t$. Thus, under Assumption~\ref{asm1}, $g_{i}(v_{i}(t))$ and $\nabla g_{i}(v_{i}(t))$ are also uniformly bounded. Let
\begin{equation*}
G_{i}=\{x\mid\nabla g_{i}(x)=0\},\quad G_{i0}=\{x\mid g_{i}(x)<0\}.
\end{equation*}
Since $g_{i}(x)$ is convex, $G_{i}$ is a closed and convex set and $G_{i0}$ is an open set. 
Moreover, under Assumption~\ref{asm3}, there holds $\min\limits_{x}g_{i}(x)<0$ for all $i=1,2,\cdots,N$, which implies that $G_{i0}$ contains $G_{i}$ strictly. Then, noting that $\nabla g_{i}(x)$ is continuous, we can conclude that the set $\{\|\nabla g_{i}(x)\|\mid x\notin G_{i0}\}$ has a strictly positive infimum, which means that there exists a positive constant $\underline{M}\leq \|d_{0}\|$ such that $\|\nabla g_{i}(x)\|\geq\underline{M}$ holds for all $x\notin G_{i0}$, i.e., for all $t\geq0$, there holds $\|d_{i}(t)\|\geq\underline{M}$. Taking into account that $g_{i}(x_{i}(t))$ and $\nabla g_{i}(x_{i}(t))$ are also uniformly bounded, we have that $d_{i}(t)$ and $k_{i}(t)$ are also uniformly bounded. Hence, for convenience of the subsequent analysis, we can assume that for all $i=1,2,\cdots,N$,
\begin{equation*}
\max\big\{\|g_{i}(v_{i}(t))\|,\|\nabla g_{i}(v_{i}(t))\|,\|d_{i}(t)\|,\|k_{i}(t)\|\big\}\leq M.
\end{equation*}
\end{remark}

Next, to further characterize the error term $\|y_{i}(t)-\mu^{i}(t-1)\tilde{y}(t)\|$, we introduce another two new variables:
\begin{equation*}
w_{1}(t)=\sum\limits_{i=1}^{N}\|\phi_{i}(t)\|,\quad w_{2}(t)=\sum\limits_{i=1}^{N}g_{i}^{+}(z_{i}(t)),
\end{equation*}
with
$\phi_{i}(t)=x_{i}(t+1)-(v_{i}(t)-\beta k_{i}(t))$. 
Besides, define
\begin{align*}
\gamma_{i}(t)=\alpha(t)\sum\limits_{s=0}^{t-1}\lambda_{1}^{t-1-s}w_{i}(s),\ i=1,2,
\end{align*}
with $\gamma_{1}(0)=\gamma_{2}(0)=0$. In the following lemma, we will deduce a common bound for all the terms $\alpha(t)\|x_{i}(t)-\bar{x}(t)\|$.

\begin{lemma}\label{lemma9}
Let Assumptions~\ref{asm3}, \ref{asm1}, \ref{asm2}, and \ref{asm4} hold. Then, for the algorithm (\ref{algorithm1_eq1}) and all $i=1,2,\cdots,N$, there holds
\begingroup
\allowdisplaybreaks
\begin{align}\label{lemma9_eq1}
\alpha(t)\|x_{i}(t)-\bar{x}(t)\|
&\leq C_{6}\gamma_{1}(t)+C_{7}\gamma_{2}(t)+\xi_{1}(t),
\end{align}
\endgroup%
where $C_{6}=2N^{2}C_{1}$, $C_{7}=2N^{2}C_{1}\frac{\beta}{\underline{M}}$, and
\begin{align*}
\xi_{1}(t)&=C_{8}\alpha(t)\sum\limits_{s=1}^{t-1}\lambda_{1}^{t-1-s}\lambda_{2}^{s-1}+C_{10}\alpha(t)\sum\limits_{s=1}^{t-1}\lambda_{1}^{t-1-s}\alpha(s)\\
      &\veq +C_{5}\alpha(t)\lambda_{1}^{t-1}+C_{9}\alpha(t)\sum\limits_{s=1}^{t-1}\lambda_{1}^{t-1-s}\sum\limits_{r=0}^{s-1}\lambda_{2}^{s-1-r}\alpha(r),
\end{align*}
with
\begin{align*}
C_{5}&=NC_{1}\sum\limits_{j=1}^{N}\|x_{j}(0)-\bar{x}(0)\|+2N^{2}C_{1}(w_{1}(0)+\sum\limits_{j=1}^{N}y_{j}(0)+\beta\sum\limits_{j=1}^{N}\|k_{j}(0)\|),\\
C_{8}&=2N^{2}C_{1}C_{3}\Big(1+\frac{\beta M}{\underline{M}}\Big),\
C_{9}=2N^{2}C_{1}C_{4}\Big(1+\frac{\beta M}{\underline{M}}\Big),C_{10}=2N^{3}MC_{1}\Big(1+\frac{\beta M}{\underline{M}}\Big).
\end{align*}
\end{lemma}

Based on Lemma~\ref{lemma9}, we will provide a more concrete bound of the term $\|y_{i}(t)-\mu^{i}(t-1)\tilde{y}(t)\|$ in the next lemma. To this end, we need to introduce several intermediate variables:
\begin{align*}
\eta_{i}(t)=\sum\limits_{s=0}^{t-1}\lambda_{2}^{t-1-s}\gamma_{i}(s),\ i=1,2,\ \eta_{3}(t)=\sum\limits_{s=0}^{t-1}\lambda_{2}^{t-1-s}\alpha(s)w_{2}(s),
\end{align*}
with $\eta_{1}(0)=\eta_{2}(0)=\eta_{3}(0)=0$.

\begin{lemma}\label{lemma10}
Let Assumptions~\ref{asm3}, \ref{asm1}, \ref{asm2}, and~\ref{asm4} hold. Then, for the algorithm (\ref{algorithm1_eq1}) and all $i=1,2,\cdots,N$, the following inequality holds:
\begingroup
\allowdisplaybreaks
\begin{align}\label{lemma10_eq1}
&\veq\|y_{i}(t)-\mu^{i}(t-1)\tilde{y}(t)\|\notag\\
&\leq C_{3}\lambda_{2}^{t-1}+C_{17}\sum\limits_{s=0}^{t-1}\lambda_{2}^{t-1-s}\xi_{2}(s)
+C_{18}\eta_{1}(t)+C_{19}\eta_{2}(t)+C_{20}\eta_{3}(t),
\end{align}
\endgroup
where
\begingroup
\allowdisplaybreaks
\begin{align*}
\xi_{2}(t)
&=M(\alpha(t)\!-\!\alpha(t+1))\!+\!L(1\!+\!N)\xi_{1}(t)\!+\!C_{13}\alpha(t)\lambda_{2}^{t-1}\notag\\
&\veq+C_{14}\alpha(t)\sum\limits_{s=0}^{t-1}\lambda_{2}^{t-1-s}\alpha(s)+C_{15}\alpha^{2}(t),
\end{align*}
\endgroup
with $C_{11}=L(1+N)C_{6}$, $C_{12}=L(1+N)C_{7}$,
\begin{align*}
C_{13}&=L\Big(1+\frac{\beta M}{\underline{M}}\Big)C_{3},\
C_{14}=L\Big(1+\frac{\beta M}{\underline{M}}\Big)C_{4},\
C_{15}=LNM\Big(1+\frac{\beta M}{\underline{M}}\Big),\\
C_{16}&=\frac{\beta L}{\underline{M}},\
C_{17}=NC^{\prime}_{4},\
C_{18}=NC_{11}C^{\prime}_{4},\ C_{19}=NC_{12}C^{\prime}_{4},\ C_{20}=NC_{16}C^{\prime}_{4}.
\end{align*}
\end{lemma}

Now, we are prepared to examine the evolution of all states $x_{i}(t)$ in the below lemma.

\begin{lemma}\label{lemma11}
Under Assumptions~\ref{asm3}, \ref{asm1}, \ref{asm2}, and~\ref{asm4}, for all $v\in X^{*}$ and $t\geq0$, there holds
\begingroup
\allowdisplaybreaks
\begin{align}\label{lemma11_eq1}
&\veq\sum\limits_{i=1}^{N}\pi^{i}(t+1)\|x_{i}(t+1)-v\|^{2}\notag\\
&\leq\sum\limits_{i=1}^{N}\pi^{i}(t)\|x_{i}(t)-v\|^{2}-2\theta_{2}\alpha(t)(f(s(t))-f(v))+\xi_{3}(t)+D_{4}\gamma_{1}(t)+D_{5}\gamma_{2}(t)\notag\\
&\veq-\sum\limits_{i=1}^{N}\frac{\theta_{1}\beta(2-\beta)}{2 M^{2}}(g_{i}^{+}(z_{i}(t)))^{2}
+D_{6}\eta_{1}(t)+D_{7}\eta_{2}(t)+D_{8}\eta_{3}(t)-\sum\limits_{i=1}^{N}\theta_{1}\|\phi_{i}(t)\|^{2},
\end{align}
\endgroup
where $D_{4}=2C_{6}D_{1}N$, $D_{5}=2C_{7}D_{1}N$, $D_{6}=C_{18}D_{3}N$, $D_{7}=C_{19}D_{3}N$, $D_{8}=C_{20}D_{3}N$, and
\begingroup
\allowdisplaybreaks
\begin{align*}
\xi_{3}(t)
=2D_{1}N\xi_{1}(t)+D_{2}\alpha^{2}(t)+C_{3}D_{3}N\lambda_{2}^{t-1}+C_{17}D_{3}N\sum\limits_{s=0}^{t-1}\lambda_{2}^{t-1-s}\xi_{2}(s),
\end{align*}
\endgroup
with $D_{1}=MN+LM+(MN+1)MNR$, $D_{3}=2M+8M\beta(2-\beta)$, and
\begin{align*}
D_{2}=\frac{4M^{4}N^{3}R^{2}}{\theta_{1}\beta(2-\beta)}+2M^{2}N^{2}+8M^{2}N^{2}\beta(2-\beta).
\end{align*}
\end{lemma}

It should be noted that in Lemma~\ref{lemma11}, the summability of $\gamma_{i}(t)$, $i=1,2$, and $\eta_{j}(t),$ $j=1,2,3$, is still unknown for us. Hence, these terms need to be further analyzed.
Some properties of $\gamma_{i}(t)$, $i=1,2$, and $\eta_{j}(t),$ $j=1,2,3$, will be revealed in the following proposition whose proof is omitted for brevity.

\begin{proposition}\label{pro1}
For all $t\geq0$, the below inequalities hold:
\begingroup
\allowdisplaybreaks
\begin{subequations}\label{pro1_eq1}
\begin{align}
\gamma_{i}(t+1)&\leq\lambda_{1}\gamma_{i}(t)+\alpha(t)w_{i}(t),\ i=1,2,\label{pro1_eq1a}\\
\eta_{i}(t+1)&\leq\lambda_{2}\eta_{i}(t)+\gamma_{i}(t),\ i=1,2, \label{pro1_eq1b}\\
\eta_{3}(t+1)&\leq\lambda_{2}\eta_{3}(t)+\alpha(t)w_{2}(t).\label{pro1_eq1c}
\end{align}
\end{subequations}
\endgroup
\end{proposition}

To continue, based on the results in Proposition~\ref{pro1} and \cite{Mai2019Aut}, we give another key lemma to
further characterize the evolution of all states $x_{i}(t)$.

\begin{lemma}\label{lemma12}
Under Assumptions~\ref{asm3}, \ref{asm1}, \ref{asm2}, and~\ref{asm4}, for all $v\in X^{*}$ and $t\geq0$, there holds
\begingroup
\allowdisplaybreaks
\begin{align}\label{lemma12_eq1}
&\veq\sum\limits_{i=1}^{N}\pi^{i}(t+1)\|x_{i}(t+1)-v\|^{2}+e(t+1)\notag\\
&\leq\sum\limits_{i=1}^{N}\pi^{i}(t)\|x_{i}(t)-v\|^{2}+e(t)+2\theta_{2}\alpha(t)(f(v)-f(s(t))\notag\\
&\veq+\xi_{4}(t)-\frac{\theta_{1}}{2}\sum\limits_{i=1}^{N}\|\phi_{i}(t)\|^{2}-\frac{\theta_{1}\beta(2-\beta)}{4 M^{2}}\sum\limits_{i=1}^{N}(g_{i}^{+}(z_{i}(t)))^{2},
\end{align}
\endgroup
where $e(t)=\sum\limits_{i=1}^{2}(a_{i}b_{i}\gamma_{i}(t)+c_{i}\eta_{i}(t))+a_{3}b_{3}\eta_{3}(t)$ and
$\xi_{4}(t)=\xi_{3}(t)+\frac{1}{2}(a_{1}^{2}+a_{2}^{2}+a_{3}^{2})\alpha^{2}(t)$
with
\begin{align*}
a_{i}&=\frac{c_{i}+D_{3+i}}{(1-\lambda_{2})b_{i}},\ i=1,2,\
a_{3}=\frac{D_{8}}{(1-\lambda_{2})b_{3}},\
b_{1}=\sqrt{\frac{\theta_{1}}{N}},\\
b_{2}&=b_{3}=\sqrt{\frac{\theta_{1}\beta(2-\beta)}{4M^{2}N}},\
c_{i}=\frac{D_{5+i}}{1-\lambda_{2}},\ i=1,2.
\end{align*}
\end{lemma}

\subsection{Proof of Theorem~\ref{theorem1}}\label{sec:PT-Thm1}

In this subsection, based on the above intermediate results and the proof of \cite[Theorem 4]{Mai2019Aut}, we are now ready to present the detailed proof of Theorem~\ref{theorem1} regarding the convergence property of the algorithm (\ref{algorithm1_eq1}).

\begin{IEEEproof}[\textbf{Proof~of~Theorem~\ref{theorem1}}]
First, the nonnegative sequences $\{a(t)\}$, $\{b(t)\}$, $\{c(t)\}$, and $\{d(t)\}$ in Lemma~\ref{lemma3} can be defined as follows:
\begingroup
\allowdisplaybreaks
\begin{align*}\label{theorem1_proof_eq1}
a(t)&=\sum\limits_{i=1}^{N}\pi^{i}(t)\|x_{i}(t)-v\|^{2}+e(t),\ b(t)=0,\ c(t)=\xi_{4}(t),\notag\\
d(t)&=2\theta_{2}\alpha(t)(f(s(t))-f(v))+\frac{\theta_{1}}{2}\sum\limits_{i=1}^{N}\|\phi_{i}(t)\|^{2}+\frac{\theta_{1}\beta(2-\beta)}{4 M^{2}}\sum\limits_{i=1}^{N}(g_{i}^{+}(z_{i}(t)))^{2}.
\end{align*}
\endgroup%
Then, noting that $\alpha^{2}(t)$ is summable, we can easily obtain from part (b) of Lemma~\ref{lemma2} that $\xi_{1}(t)$ is summable. In a similar manner, we can get that $\xi_{2}(t)$ and $\xi_{3}(t)$ are also summable, which directly implies the summability of $\xi_{4}(t)$ (i.e., $c(t)$). So by Lemma~\ref{lemma3}, there exists a nonnegative constant $\delta_{0}$ such that
\begin{equation}\label{theorem1_proof_eq2}
\lim\limits_{t\to\infty}\bigg(\sum\limits_{i=1}^{N}\pi^{i}(t)\|x_{i}(t)-v\|^{2}+e(t)\bigg)=\delta_{0},
\end{equation}
\begin{equation}\label{theorem1_proof_eq3}
\sum\limits_{t=0}^{\infty}\alpha(t)(f(s(t))-f(v))<\infty,
\end{equation}
\vspace{-1ex}
\begin{equation}\label{theorem1_proof_eq4}
\sum\limits_{t=0}^{\infty}\bigg[\frac{\theta_{1}}{2}\sum\limits_{i=1}^{N}\|\phi_{i}(t)\|^{2}+\frac{\theta_{1}\beta(2-\beta)}{4 M^{2}}\sum\limits_{i=1}^{N}(g_{i}^{+}(z_{i}(t)))^{2}\bigg]<\infty.
\end{equation}

Clearly, it follows from (\ref{theorem1_proof_eq4}) that
\begin{equation}
\lim\limits_{t\to\infty}\sum\limits_{i=1}^{N}\|\phi_{i}(t)\|^{2}=0,\ \
\label{theorem1_proof_eq6}
\lim\limits_{t\to\infty}\sum\limits_{i=1}^{N}(g_{i}^{+}(z_{i}(t)))^{2}=0,
\end{equation}
which yields that $\lim\limits_{t\to\infty}\gamma_{i}(t)=0,\ i=1,2,$ and $\lim\limits_{t\to\infty}\eta_{i}(t)=0,\ i=1,2,3,$ according to part (a) of Lemma~\ref{lemma2}. Thus, we have $\lim\limits_{t\to\infty}e(t)=0$, and from (\ref{lemma9_proof_eq7}) to be given in the proof of Lemma~\ref{lemma9} in the Appendix, we can conclude that
\begin{equation}\label{theorem1_proof_eq7}
\lim\limits_{t\to\infty}\|x_{i}(t)-\bar{x}(t)\|=0.
\end{equation}
Furthermore, considering (\ref{theorem1_proof_eq2}), we have 
\begin{equation*}\label{theorem1_proof_eq8}
\lim\limits_{t\to\infty}\sum\limits_{i=1}^{N}\pi^{i}(t)\|x_{i}(t)-v\|^{2}=\delta_{0}.
\end{equation*}

Given that $\sum\limits_{t=0}^{\infty}\alpha(t)=\infty$, we can get from (\ref{theorem1_proof_eq3}) that $\liminf\limits_{t\to\infty}f(s(t))=f(v)$. As $\{s(t)\}$ is contained in the compact set $X$, it has a convergent subsequence $\{s(t_{k})\}$, that is, there exists $s^{*}$ such that $\lim\limits_{k\to\infty}s(t_{k})=s^{*}$. Moreover, since $X_{0}$ is a closed set, we have $s^{*}\in X_{0}$. Thus, it follows from the continuity property of $f$ that $\lim\limits_{t\to\infty}f(s(t_{k}))=f(s^{*})=f(v)$,
which implies that $s^{*}\in X^{*}$. Since $v$ is arbitrarily chosen in $X^{*}$, we can substitute $v$ by $s^{*}$. Then, it can be verified that $\delta_{0}=0$ when $v=s^{*}$. So we have 
\begin{equation*}
\|x_{i}(t)\!-\!s^{*}\|^{2}\leq3(\|x_{i}(t)\!-\!\bar{x}(t)\|^{2}\!+\!\|\bar{x}(t)\!-\!s(t)\|^{2}\!+\!\|s(t)\!-\!s^{*}\|^{2}).
\end{equation*}

Utilizing
(\ref{lemma11_proof_eq6}) to be given in the proof of Lemma~\ref{lemma11} in the Appendix, we arrive at
\begingroup
\allowdisplaybreaks
\begin{align*}
\|\bar{x}(t)-s(t)\|^{2}&\leq 2R^{2}(NM+1)^{2}N\sum\limits_{i=1}^{N}\|x_{i}(t)-\bar{x}(t)\|^{2}+2R^{2}N\sum\limits_{j=1}^{N}(g_{i}^{+}(z_{i}(t)))^{2}.
\end{align*}
\endgroup
As a result, we attain
\begin{align}\label{theorem1_proof_eq10}
\frac{1}{3}\|x_{i}(t)-s^{*}\|^{2}
&\leq\|x_{i}(t)-\bar{x}(t)\|^{2}+2R^{2}(NM+1)^{2}N\sum\limits_{i=1}^{N}\|x_{i}(t)-\bar{x}(t)\|^{2}\notag\\
&\veq+2R^{2}N\sum\limits_{i=1}^{N}(g_{i}^{+}(z_{i}(t)))^{2}+\|s(t)-s^{*}\|^{2}.
\end{align}
Multiplying both sides of the above inequality (\ref{theorem1_proof_eq10}) by $\pi^{i}(t)$ and summing over from $i=1$ to $i=N$ produces
\begingroup
\allowdisplaybreaks
\begin{align}\label{theorem1_proof_eq11}
\sum\limits_{i=1}^{N}\frac{\pi^{i}(t)}{3}\|x_{i}(t)-s^{*}\|^{2}
&\leq\sum\limits_{i=1}^{N}\pi^{i}(t)\|x_{i}(t)-\bar{x}(t)\|^{2}+\sum\limits_{i=1}^{N}\pi^{i}(t)\|s(t)-s^{*}\|^{2}\notag\\
&\veq+2R^{2}(M+1)^{2}N\sum\limits_{i=1}^{N}\pi^{i}(t)\sum\limits_{i=1}^{N}\|x_{i}(t)-\bar{x}(t)\|^{2}\notag\\
&\veq+2R^{2}N\sum\limits_{i=1}^{N}\pi^{i}(t)\sum\limits_{i=1}^{N}(g_{i}^{+}(z_{i}(t)))^{2}.
\end{align}
\endgroup
Then, taking $\liminf\limits_{t\to\infty}$ for both sides of (\ref{theorem1_proof_eq11}) and using (\ref{theorem1_proof_eq6}) as well as (\ref{theorem1_proof_eq7}) yields $\frac{\delta_{0}}{3}\leq\liminf_{t\to\infty}\|s(t)-s^{*}\|$.

Finally, since $\lim\limits_{k\to\infty}s(t_{k})=s^{*}$, we have 
\begin{equation*}
\liminf\limits_{t\to\infty}\|s(t)-s^{*}\|=0,
\end{equation*}
which implies that $\delta_{0}=0$. Noting that $\pi^{i}(t)\geq\theta_{1}>0$ for all $i=1,\cdots,N$, and $t\geq0$, we can get that $\lim\limits_{t\to\infty}x_{i}(t)=s^{*}$ for all $i=1,\cdots,N$, thus establishing the convergence of the algorithm (\ref{algorithm1_eq1}). The proof is complete.
\end{IEEEproof}

\subsection{Proof of Theorem~\ref{theorem2}}\label{sec:PT-Thm2}

In this subsection, with the inspirations from the analysis of the convergence rate shown in \cite{Nedic2015TAC} and \cite{Mai2019Aut} and the above intermediate results in Subsection~\ref{sec:PT-Pre}, we will give the detailed proof of Theorem~\ref{theorem2} about the convergence rate of the
algorithm (\ref{algorithm1_eq1}).

\begin{IEEEproof}[\textbf{Proof~of~Theorem~\ref{theorem2}}]
First, according to (\ref{lemma12_eq1}), by defining
\begin{align*}
\varphi_{1}(t)=\frac{\theta_{1}}{4\theta_{2}}\sum\limits_{i=1}^{N}\|\phi_{i}(t)\|^{2},\
\varphi_{2}(t)=\frac{\theta_{1}\beta(2-\beta)}{8 \theta_{2}M^{2}}\sum\limits_{i=1}^{N}(g_{i}^{+}(z_{i}(t)))^{2},
\end{align*}
we have
\begingroup
\allowdisplaybreaks
\begin{align}\label{them2_proof_eq1}
&\veq\sum\limits_{k=0}^{t}\bigg[\alpha(k)(f(s(k))-f^{*})+\varphi_{1}(k)+\varphi_{2}(k)\bigg]\notag\\
&\leq \frac{1}{2\theta_{2}}\Big(\sum\limits_{i=1}^{N}\pi^{i}(0)\|x_{i}(0)-v\|^{2}+e(0)+\sum\limits_{k=0}^{t}\xi_{4}(k)\Big).
\end{align}
\endgroup
Obviously, for any nonnegative summable scalar sequence $\{h(t)\}$ and any constant $0<\rho<1$, there holds
\begingroup
\allowdisplaybreaks
\begin{align}\label{them2_proof_eq2}
\sum\limits_{k=0}^{t}\sum\limits_{s=0}^{k}\rho^{k-s}h(s)=\sum\limits_{s=0}^{t}h(s)\sum\limits_{k=s}^{t}\rho^{k-s}\leq\sum\limits_{k=0}^{t}\frac{h(k)}{1-\rho},
\end{align}
\endgroup
which implies that
\begingroup
\allowdisplaybreaks
\begin{align}\label{them2_proof_eq3}
\sum\limits_{k=0}^{t}\xi_{1}(k)
&\leq\frac{C_{8}\alpha(0)}{(1-\lambda_{1})(1-\lambda_{2})}+\frac{C_{10}}{1-\lambda_{1}}\sum\limits_{k=0}^{t}\alpha^{2}(k)\notag\\
&\veq+\frac{C_{5}\alpha(0)}{1-\lambda_{1}}+\frac{C_{9}}{(1-\lambda_{1})(1-\lambda_{2})}\sum\limits_{k=0}^{t}\alpha^{2}(k)\notag\\
&=M_{1}+M_{2}\sum\limits_{k=0}^{t}\alpha^{2}(k),\\
\label{them2_proof_eq4}
\sum\limits_{k=0}^{t}\xi_{2}(k)
&\leq M\alpha(0)+L(1+N)\Big(M_{1}+M_{2}\sum\limits_{k=0}^{t}\alpha^{2}(k)\Big)\notag\\
&\veq+\frac{C_{13}\alpha(0)}{1-\lambda_{2}}\!+\!\frac{C_{14}}{1-\lambda_{2}}\!\sum\limits_{k=0}^{t}\!\alpha^{2}(k)
\!+\!C_{15}\!\sum\limits_{k=0}^{t}\!\alpha^{2}(k)\notag\\
&=M_{3}+M_{4}\sum\limits_{k=0}^{t}\alpha^{2}(k),\\
\label{them2_proof_eq5}
\sum\limits_{k=0}^{t}\xi_{3}(k)
&\leq 2ND_{1}\Big(M_{1}+M_{2}\sum\limits_{k=0}^{t}\alpha^{2}(k)\Big)+D_{2}\sum\limits_{k=0}^{t}\alpha^{2}(k)\notag\\
&\veq+\frac{C_{3}D_{3}N}{1-\lambda_{2}}+\frac{C_{17}D_{3}N}{1-\lambda_{2}}\Big(M_{3}+M_{4}\sum\limits_{k=0}^{t}\alpha^{2}(k)\Big)\notag\\
&=M_{5}+M_{6}\sum\limits_{k=0}^{t}\alpha^{2}(k),\\
\label{them2_proof_eq6}
\sum\limits_{k=0}^{t}\xi_{4}(k)
&\leq M_{5}+M_{7}\sum\limits_{k=0}^{t}\alpha^{2}(k),
\end{align}
\endgroup
where
\begingroup
\allowdisplaybreaks
\begin{align*}
M_{1}&=\frac{D_{8}\alpha(0)}{(1-\lambda_{1})(1-\lambda_{2})}+\frac{C_{5}\alpha(0)}{1-\lambda_{1}},\
M_{2}=\frac{C_{10}}{1-\lambda_{1}}+\frac{C_{9}}{(1-\lambda_{1})(1-\lambda_{2})},\notag\\
M_{3}&=M\alpha(0)+L(1+N)M_{1}+\frac{C_{13}\alpha(0)}{1-\lambda_{2}},\
M_{4}=L(1+N)M_{2}+\frac{C_{14}}{1-\lambda_{2}}+C_{15},\notag\\
M_{5}&=2D_{1}M_{1}N+\frac{C_{3}D_{3}N}{1-\lambda_{2}}+\frac{C_{17}D_{3}M_{3}N}{1-\lambda_{2}},\
M_{6}=2D_{1}M_{2}N+D_{2}+\frac{C_{17}D_{3}M_{4}N}{1-\lambda_{2}},\notag\\
M_{7}&=M_{6}+\frac{1}{2}(a_{1}^{2}+a_{2}^{2}+a^{2}_{3}).
\end{align*}
\endgroup
Thus, it follows from (\ref{them2_proof_eq1}) that
\begingroup
\allowdisplaybreaks
\begin{align}\label{them2_proof_eq7}
\sum\limits_{k=0}^{t}\bigg[\alpha(k)(f(s(k))-f^{*})+\varphi_{1}(k)+\varphi_{2}(k)\bigg]\leq M_{8}+M_{9}\sum\limits_{k=0}^{t}\alpha^{2}(k),
\end{align}
\endgroup
with
\vspace{-1ex}
\begin{align*}
M_{8}=\frac{1}{2\theta_{2}}\!\bigg(\!\sum\limits_{i=1}^{N}\!\pi^{i}(t)\|x_{i}(0)\!-\!v\|^{2}\!+\!e(0)\!+\!M_{5}\!\bigg)\!,\;
M_{9}=\frac{M_{7}}{2\theta_{2}}.
\end{align*}

Additionally, according to (\ref{lemma11_proof_eq6}) to be given in the proof of Lemma~\ref{lemma11} in the Appendix, we can get
\begingroup
\allowdisplaybreaks
\begin{align}\label{them2_proof_eq8}
&\veq\|x_{i}(t)-s(t)\|\notag\\
&\leq \|x_{i}(t)-\bar{x}(t)\|+\|\bar{x}(t)-s(t)\|\notag\\
&\leq \|x_{i}(t)-\bar{x}(t)\|+R\sum\limits_{i=1}^{N}g_{i}^{+}(z_{i}(t))+R(MN+1)\sum\limits_{i=1}^{N}\|x_{i}(t)-\bar{x}(t)\|.
\end{align}
\endgroup
Thus, by Lemma~\ref{lemma9}, we have
\begingroup
\allowdisplaybreaks
\begin{align}\label{them2_proof_eq9}
\frac{\alpha(t)\|x_{i}(t)-s(t)\|}{\kappa(RMN^{2}+RN+1)}
&\leq\gamma_{1}(t)+\gamma_{2}(t)+\frac{1}{\kappa}\xi_{1}(t)+\frac{R\alpha(t)w_{2}(t)}{\kappa(RMN^{2}+RN+1)},
\end{align}
\endgroup
with $\kappa=\max\{C_{6},C_{7}\}$. Hence, combining (\ref{them2_proof_eq3}) and (\ref{them2_proof_eq9}) together yields
\begingroup
\allowdisplaybreaks
\begin{align}\label{them2_proof_eq10}
&\veq\frac{1}{\kappa(RMN^{2}+RN+1)}\sum\limits_{k=0}^{t}\alpha(k)\|x_{i}(k)-s(k)\|\notag\\
&\leq\sum\limits_{k=0}^{t}\gamma_{1}(k)+\sum\limits_{k=0}^{t}\gamma_{2}(k)+\frac{1}{\kappa}\Big(M_{1}+M_{2}\sum\limits_{k=0}^{t}\alpha^{2}(k)\Big)\notag\\
&\veq+\frac{R}{\kappa(RMN^{2}+RN+1)}\sum\limits_{k=0}^{t}\alpha(k)w_{2}(t).
\end{align}
\endgroup
Clearly, we have
\begingroup
\allowdisplaybreaks
\begin{align}\label{them2_proof_eq11}
\sum\limits_{k=0}^{t}\gamma_{i}(k)
\leq\lambda_{1}\sum\limits_{k=0}^{t-1}\gamma_{i}(k)+\sum\limits_{k=0}^{t-1}\alpha(k)w_{i}(k)
\leq\lambda_{1}\sum\limits_{k=0}^{t}\gamma_{i}(k)+\sum\limits_{k=0}^{t}\alpha(k)w_{i}(k),
\end{align}
\endgroup
which implies that
\vspace{-1ex}
\begingroup
\allowdisplaybreaks
\begin{align}\label{them2_proof_eq12}
\sum\limits_{k=0}^{t}\gamma_{i}(k)
\leq\frac{1}{1-\lambda_{1}}\sum\limits_{k=0}^{t}\alpha(k)w_{i}(k)
\leq\frac{1}{2(1-\lambda_{1})}\sum\limits_{k=0}^{t}\alpha^{2}(k)+\frac{1}{2(1-\lambda_{1})}\sum\limits_{k=0}^{t}w^{2}_{i}(k).
\end{align}
\endgroup

\vspace{-1ex}
\noindent
Thus, it follows from (\ref{them2_proof_eq10}) that
\vspace{-1ex}
\begingroup
\allowdisplaybreaks
\begin{align}\label{them2_proof_eq13}
&\veq\frac{2(1-\lambda_{1})}{\kappa(RMN^{2}+RN+1)}\sum\limits_{k=0}^{t}\alpha(k)\|x_{i}(k)-s(k)\|\notag\\
&\leq\sum\limits_{k=0}^{t}\alpha^{2}(k)+\sum\limits_{k=0}^{t}w^{2}_{1}(k)+\sum\limits_{k=0}^{t}\alpha^{2}(k)+\sum\limits_{k=0}^{t}w^{2}_{2}(k)+\sum\limits_{k=0}^{t}w^{2}_{2}(k)\notag\\
&\veq+\frac{2(1-\lambda_{1})}{\kappa}\Big(M_{1}+M_{2}\sum\limits_{k=0}^{t}\alpha^{2}(k)\Big)
+\frac{R^{2}(1-\lambda_{1})^{2}}{\kappa^{2}(RMN^{2}+RN+1)^{2}}\sum\limits_{k=0}^{t}\alpha^{2}(k).
\end{align}
\endgroup

\vspace{-1ex}
Now, introducing $\psi_{1}(t)=\sum\limits_{i=1}^{N}\|\phi_{i}(t)\|^{2}$ and $\psi_{2}(t)=\sum\limits_{i=1}^{N}(g_{i}^{+}(z_{i}(t)))^{2}$,
and noting that
$w_{i}^{2}(t)\leq N\psi_{i}(t)$, $i=1,2$, we obtain
\vspace{-1ex}
\begingroup
\allowdisplaybreaks
\begin{align}\label{them2_proof_eq14}
H_{1}\sum\limits_{k=0}^{t}\alpha(k)\|x_{i}(k)-s(k)\|
\leq M_{10}+M_{11}\sum\limits_{k=0}^{t}\alpha^{2}(k)+\delta(\psi_{1}(t)+\psi_{2}(t)),
\end{align}
\endgroup
where $\delta=\min\Big\{\frac{\theta_{1}}{4\theta_{2}},\frac{\theta_{1}\beta(2-\beta)}{8 \theta_{2}M^{2}}\Big\}$,
\vspace{-0.5ex}
\begin{align*}
H_{1}&=\frac{\delta(1-\lambda_{1})}{\kappa(RMN^{2}+RN+1)N},\ M_{10}=\frac{(1-\lambda_{1})M_{1}\delta}{\kappa N},\\
M_{11}&=\frac{\delta}{N}+\frac{(1-\lambda_{1})M_{2}\delta}{\kappa N}+\frac{R^{2}(1-\lambda_{1})^{2}\delta}{2\kappa^{2}(RMN^{2}+RN+1)^{2}N}.
\end{align*}
Then, it follows from (\ref{them2_proof_eq7}) and (\ref{them2_proof_eq14}) that
\vspace{-1ex}
\begingroup
\allowdisplaybreaks
\begin{align}\label{them2_proof_eq15}
\sum\limits_{k=0}^{t}[H_{1}\alpha(k)\|x_{i}(k)-s(k)\|+\alpha(k)(f(s(k))-f^{*})]
\leq H_{2}+H_{3}\sum\limits_{k=0}^{t}\alpha^{2}(k),
\end{align}
\endgroup
where $H_{2}=M_{8}+M_{10}$ and $H_{3}=M_{9}+M_{11}$.

Since $f(x)$ and the norm function are both convex, we have
\begingroup
\allowdisplaybreaks
\begin{align}
\label{them2_proof_eq16}
f(\tilde{s}(t))-f^{*}&\leq\sum\limits_{k=0}^{t}\frac{\alpha(k)(f(s(k))-f^{*})}{\sum\limits_{k=0}^{t}\alpha(k)},\\
\label{them2_proof_eq17}
\|\tilde{x}_{i}(t)-\tilde{s}(t)\|&\leq\sum\limits_{k=0}^{t}\frac{\alpha(k)\|x_{i}(k)-s(k)\|}{\sum\limits_{k=0}^{t}\alpha(k)}.
\end{align}
\endgroup
Dividing both sides of (\ref{them2_proof_eq15}) by $\sum\limits_{k=0}^{t}\alpha(k)$, we can get (\ref{them2_eq1}) from (\ref{them2_proof_eq16}) and (\ref{them2_proof_eq17}), with $L(t)$ being as given in (\ref{them2_eq2}). Furthermore, it is straightforward to deduce
\begingroup
\allowdisplaybreaks
\begin{align}\label{them2_proof_eq18}
|f(\tilde{x}_{i}(t))-f^{*}|
&\leq|f(\tilde{x}_{i}(t))-f(\tilde{s}(t))|+|f(\tilde{s}(t))-f^{*}|\notag\\
&\leq MN\|\tilde{x}_{i}(t)-\tilde{s}(t)\|+|f(\tilde{s}(t))-f^{*}|\notag\\
&\leq\frac{MN}{H_{1}}L(t)+L(t)=\bigg(\frac{MN}{H_{1}}+1\bigg)L(t).
\end{align}
\endgroup
We have thus completed the proof of the convergence rate.
\end{IEEEproof}

\section{Computer Simulations}\label{sec:simu}


Similar to the numerical example shown in \cite{Mai2019Aut}, we consider a constrained optimization problem involved in the machine learning problems with its global objective function defined on $\mathbb{R}^{3}$ as following:
\vspace{-1ex}
\begin{subequations}\label{eqn:f}
\begin{equation}\label{eqn:f-a}
\min_{x\in X}\ f(x)=\sum_{i=1}^{N}f_{i}(x),
\end{equation}
where
\begin{equation}\label{eqn:f-b}
f_{i}(x)=\ln[1+e^{-a_{i}(b_{i}^{T}\mathbf{u}+x^{3})}]+p_{i}(x),
\end{equation}%
\end{subequations}
with $p_{i}(x)=\frac{(x^{1})^{2}}{4}$ for $i=1,\cdots,r_{1}-1$, $p_{i}(x)=\frac{(x^{2})^{2}}{4}$ for $i=r_{1},\cdots,N$,
$\mathbf{u}=(x^{1},x^{2})^{T}$, $b_{i}=(0.01i,0.02i)^{T}$ are the feature vectors, and $a_{i}=(-1)^{i}$ are the corresponding labels.

\begin{figure}[t!]
\begin{center}
\includegraphics[scale=0.40]{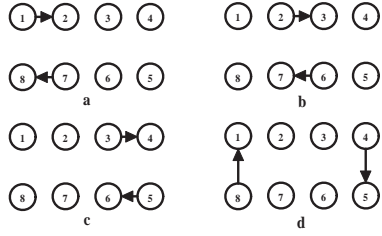}
\vspace{-1ex}
\caption{Four graphs with each having eight interacting agents and only the union of four graphs being strongly connected in Case A.}
\label{fig1}
\end{center}
\end{figure}

\textbf{\textit{Case A: $N=8$}}. In order to clearly show the effectiveness of the algorithm~(\ref{algorithm1_eq1}) in the case involving the time-varying unbalanced graphs with only the union of graphs being strongly connected, we first carry out numerical simulation when $N=8$ and $r_{1}=5$. Additionally, for $i=1,2,\cdots,N$, the inequality constraints $g_{i}(x)$ are selected as
\begin{equation}\label{eqn:gi}
g_{i}(x)=(x^{1})^{2}+ix^{2}+x^{3}-10\leq0,
\end{equation}
and the closed convex sets $X_{i}$ are chosen as
\begin{equation*}
X_{i}=\!\Big[\frac{i}{2}-3,\frac{i}{2}+1\Big]\!\!\times\!\!\Big[\frac{i}{2}-3.5,\frac{i}{2}+0.5\Big]
\!\!\times\!\!\Big[\frac{i}{2}-1,\frac{i}{2}+2.5\Big].
\end{equation*}
Thus, $X=\bigcap_{i=1}^{N} X_{i}=[1,2]\times[0.5,1]\times\{3\}$. To describe the time-varying unbalanced connections between eight nodes, four graphs with each having eight interacting agents are depicted in Fig.~\ref{fig1}, in which only the union of four graphs is strongly connected. Also, it is stipulated that the communication graph $\mathcal{G}(t)$ at step $t$ is respectively selected as Fig.~\ref{fig1}a, Fig.~\ref{fig1}b, Fig.~\ref{fig1}c, and Fig.~\ref{fig1}d when $\mod(t,4)=1$, $\mod(t,4)=2$, $\mod(t,4)=3$, and $\mod(k,4)=0$. Moreover, the step-sizes are chosen as $\alpha(t)=\frac{10^{-3}}{(t+1)^{0.6}}$ and $\beta=1$ for simulation. The transient behaviors of all states $x_{i}(t)$ are displayed in Fig.~\ref{fig2}, which clearly shows that all states $x_{i}(t)$ under the algorithm~(\ref{algorithm1_eq1}) converge to the common optimal solution $(1,0.5,3)^{T}$.

\textbf{\textit{Case A$'$: $N=8$}}. We now compare the algorithm~(\ref{algorithm1_eq1}) with Algorithm~1 developed in \cite{Mai2019Aut}. Case A$'$ is the same as Case A except for two changes that a strongly connected unbalanced graph as depicted in Fig.~\ref{fig3} is involved instead and the inequality constraints in (\ref{eqn:gi}) are removed. This is because Algorithm~1 in \cite{Mai2019Aut} is applicable only to the optimization problem subject to nonidentical local closed convex set constraints and under the strongly connected unbalanced graph. Moreover, all $x_{i}(t)$ under Algorithm~1 in \cite{Mai2019Aut} also converge to the common optimal solution $(1,0.5,3)^{T}$ within the setting as given in Case A$'$. Note that the algorithm~(\ref{algorithm1_eq1}) is still implemented under the same setting as described in Case A. The behaviors of the convergence criterion $\frac{1}{N}\sum\limits_{i=1}^{N}\frac{\|x_{i}(t)-x^{*}\|}{\|x^{*}\|}$ under the algorithm~(\ref{algorithm1_eq1}) in Case A and Algorithm~1 in \cite{Mai2019Aut} in Case A$'$ are plotted in Fig.~\ref{fig4}. It can be observed from Fig.~\ref{fig4} that these two algorithms exhibit a similar convergence rate although the extra inequality constraints in (\ref{eqn:gi}) and the time-varying unbalanced graphs in Fig.~\ref{fig1} are involved in the numerical simulation of Case A for the algorithm~(\ref{algorithm1_eq1}). This observation agrees well with the discussion in Remark~\ref{remark1}.

\begin{figure}[t!]
\begin{center}
\includegraphics[scale=0.4]{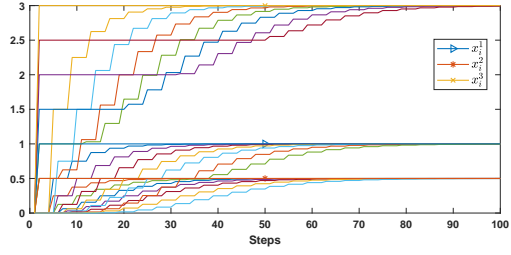}
\vspace{-1ex}
\caption{Behaviors of the states $x_{i}(t)$ under the algorithm~(\ref{algorithm1_eq1}) in Case A.}
\label{fig2}
\end{center}
\end{figure}

\begin{figure}[t!]
\begin{center}
\includegraphics[scale=0.50]{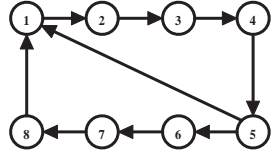}
\vspace{-1ex}
\caption{The strongly connected unbalanced graph introduced for Algorithm~1 in \cite{Mai2019Aut} with eight interacting agents in Case A$'$.}
\label{fig3}
\end{center}
\end{figure}

\begin{figure}[t!]
\begin{center}
\includegraphics[scale=0.40]{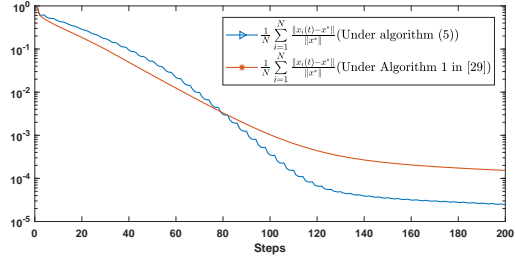}
\vspace{-1ex}
\caption{Behaviors of the convergence criterion $\frac{1}{N}\sum\limits_{i=1}^{N}\frac{\|x_{i}(t)-x^{*}\|}{\|x^{*}\|}$ under the algorithm~(\ref{algorithm1_eq1}) in Case A and Algorithm~1 in \cite{Mai2019Aut} in Case A$'$.}
\label{fig4}
\end{center}
\end{figure}

\textbf{\textit{Case B: $N=100$}}. In order to verify the effectiveness of the algorithm~(\ref{algorithm1_eq1}) in large-scale optimization problems, we reconsider the objective functions in (\ref{eqn:f}) but with $N=100$ and $r_{1}=51$. Furthermore, for $i=1,2,\cdots,N$, the inequality constraints are reselected as
\vspace{-0.5ex}
\begin{equation*}
g_{i}(x)=(x^{1})^{2}+0.1ix^{2}+x^{3}-10\leq0
\end{equation*}
and the closed convex sets are reselected as
\begin{equation*}
\begin{split}
X_{i}&=[0.06i-5,0.06i+1.94]\!\times\![0.06i-5.5,0.06i+0.94]\times[0.06i-3,0.06i+2.94].
\end{split}
\end{equation*}
Thus, $X=\bigcap\limits_{i=1}^{N} X_{i}=[1,2]\times[0.5,1]\times\{3\}$. Besides, the time-varying unbalanced directed graphs sequence with $N=100$ nodes is considered here, where the connections between nodes are time-varying in the sense that the associated iterative matrices $A(t)$ and $B(t)$ are randomly selected at each step $t$. With the same step-sizes $\alpha(t)$ and $\beta$ as chosen in Case~A, the transient behaviors of all states $x_{i}(t)$ under the algorithm~(\ref{algorithm1_eq1}) are illustrated in Fig.~\ref{fig5}, from which it is clearly seen that all states $x_{i}(t)$ also converge to the common optimal solution $x^{*}=(1,0.5,3)^{T}$.

\begin{figure}[t!]
\begin{center}
\includegraphics[scale=0.40]{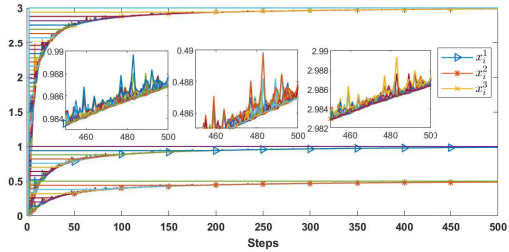}
\vspace{-1.5ex}
\caption{Behaviors of the state $x_{i}(t)$ under the algorithm~(\ref{algorithm1_eq1}) in Case B.}
\label{fig5}
\end{center}
\end{figure}

\section{Conclusion}\label{sec:conc}

In this paper, the optimization problem with nonidentical local convex inequality constraints and compact set constraints has been studied. Specifically, a distributed discrete-time algorithm has been developed over time-varying unbalanced directed topologies and its convergence property to the optimal solution has been rigorously confirmed. Moreover, the detailed analysis of the convergence rate has also been clearly shown for the proposed distributed algorithm. The advantage of the proposed algorithm is that the auxiliary variables have been introduced to estimate the gradients of the global objective function so as to offset the effect of the network-induced asymmetry. As a result, the optimization problem with multiple nonidentical local constraints has been successfully resolved over time-varying unbalanced topologies. More importantly, the work in this paper has shown an efficient mechanism for designing distributed algorithms for optimization problems with various types of constraints over time-varying unbalanced topologies. In the future, attention will be paid to addressing the distributed optimization problem formulated in this paper with a view to achieving a better convergence rate in the case strongly convex objective functions.

\section*{Appendix}\label{sec:appe}

In this Appendix, the proofs of all the intermediate results including Lemmas~\ref{lemma7}--\ref{lemma12} and Corollary~\ref{cor1} are shown in detail.

\begin{IEEEproof}[\textbf{Proof~of~Lemma~\ref{lemma7}}]
From the iteration rule (\ref{algorithm1_eq1c}) and the definition of $\tilde{y}(t)$, we have
\vspace{-1ex}
\begin{align*}
\tilde{y}(t+1)&=\sum\limits_{i=1}^{N}\sum\limits_{j=1}^{N}B_{ij}(t)y_{j}(t)+\alpha(t+1)\sum\limits_{i=1}^{N}\nabla f_{i}(x_{i}(t+1))
              -\alpha(t)\sum\limits_{i=1}^{N}\nabla f_{i}(x_{i}(t)).
\end{align*}
Considering that $B(t)$ are column stochastic matrices for all $t\geq0$, we can get
\begin{equation*}
\tilde{y}(t+1)=\tilde y(t)+\alpha(t+1)\nabla\tilde{f}(t+1)-\alpha(t)\nabla\tilde{f}(t).
\end{equation*}
Then, direct computation can yield
\begin{equation*}
\tilde{y}(t+1)-\alpha(t+1)\nabla\tilde{f}(t+1)=\tilde y(t)-\alpha(t)\nabla\tilde{f}(t),
\end{equation*}
which means that the error term $\tilde y(t)-\alpha(t)\nabla\tilde{f}(t)$ is invariant with respect to $t$, i.e., for all $t\geq0$,
\begin{equation*}
\tilde{y}(t)-\alpha(t)\nabla\tilde{f}(t)=\tilde y(0)-\alpha(0)\nabla\tilde{f}(0).
\end{equation*}
Therefore, with the given initial values, we have
\begin{equation*}
\tilde{y}(t)-\alpha(t)\nabla\tilde{f}(t)=\tilde y(0)-\alpha(0)\nabla\tilde{f}(0)=0,
\end{equation*}
which completes the proof.
\end{IEEEproof}

\begin{IEEEproof}[\textbf{Proof~of~Lemma~\ref{lemma8}}]
First, we can rewrite (\ref{algorithm1_eq1c}) as
\begin{equation*}
y_{i}(t+1)=\sum\limits_{j=1}^{N}B_{ij}(t)y_{j}(t)+\epsilon_{i}(t),
\end{equation*}
with $\epsilon_{i}(t)=\alpha(t+1)\nabla f_{i}(x_{i}(t+1))-\alpha(t)\nabla f_{i}(x_{i}(t))$ being bounded by $2\alpha(t)M$. Then, we have
\begingroup
\allowdisplaybreaks
\begin{align*}
y_{i}(t)&=\sum\limits_{j=1}^{N}B(t-1:0)_{ij}y_{j}(0)+\epsilon_{i}(t-1)
          +\sum\limits_{s=1}^{t-1}\sum\limits_{j=1}^{N}B(t-1:s)_{ij}\epsilon_{j}(s-1),\\
\tilde{y}(t)&=\tilde{y}(0)+\sum\limits_{s=0}^{t-1}\sum\limits_{j=1}^{N}\epsilon_{j}(s).
\end{align*}
\endgroup
As a result, we can obtain
\begingroup
\allowdisplaybreaks
\begin{align*}
y_{i}(t)-\mu^{i}(t-1)\tilde{y}(t)
&=\sum\limits_{j=1}^{N}D(t-1:0)_{ij}y_{j}(0)+\sum\limits_{s=1}^{t-1}\sum\limits_{j=1}^{N}D(t-1:s)_{ij}\epsilon_{j}(s-1)\notag\\
&\veq+\epsilon_{i}(t-1)-\mu^{i}(t-1)\sum\limits_{j=1}^{N}\epsilon_{j}(t-1),
\end{align*}
\endgroup
where $D(t:s)=B(t:s)-\mu(t)\mathbf{1}^{T}$. Next, from part (a) of Lemma~\ref{lemma5}, we can derive
\begingroup
\allowdisplaybreaks
\begin{align}\label{lemma8_proof_eq2}
\|y_{i}(t)-\mu^{i}(t-1)\tilde{y}(t)\|
&\leq\sum\limits_{j=1}^{N}C_{2}\lambda_{2}^{t-1}\|y_{j}(0)\|+\sum\limits_{j=1}^{N}\|\epsilon_{j}(t-1)\|\notag\\
&\veq+\sum\limits_{s=1}^{t-1}\sum\limits_{j=1}^{N}C_{2}\lambda_{2}^{t-1-s}\|\epsilon_{j}(s-1)\|\notag\\
&\leq C_{3}\lambda_{2}^{t-1}+C^{\prime}_{4}\sum\limits_{s=0}^{t-1}\sum\limits_{j=1}^{N}\lambda_{2}^{t-1-s}\|\epsilon_{j}(s)\|,
\end{align}
\endgroup
which implies (\ref{lemma8_eq1}). Thus, the proof is complete.
\end{IEEEproof}

\begin{IEEEproof}[\textbf{Proof~of~Corollary~\ref{cor1}}]
According to Lemma~\ref{lemma7}, we can directly get
\begin{align}\label{cor_proof_eq1}
\|\tilde{y}(t)\|\leq NM\alpha(t).
\end{align}
Observing that $\mu^{i}(t)<1$ and
\begin{align}\label{cor_proof_eq2}
\|y_{i}(t)\|\leq\|y_{i}(t)-\mu^{i}(t-1)\tilde{y}(t)\|+\mu^{i}(t-1)\|\tilde{y}(t)\|,
\end{align}
we can easily deduce (\ref{cor_eq1}) from (\ref{lemma8_eq1}) and (\ref{cor_proof_eq1}).

Clearly, the first term $C_{3}\lambda_{2}^{t-1}$ and the third term $NM\alpha(t)$ on the right-hand side of (\ref{cor_eq1}) converge to zero when $t$ tends to infinity. Therefore, noting from part (a) of Lemma~\ref{lemma2} that
\begin{align*}
\lim\limits_{t\to\infty}C_{4}\sum\limits_{s=0}^{t-1}\lambda_{2}^{t-1-s}\alpha(s)=0,
\end{align*}
we have that $\lim\limits_{t\to\infty}\|y_{i}(t)\|=0$, which completes the proof.
\end{IEEEproof}

\begin{IEEEproof}[\textbf{Proof~of~Lemma~\ref{lemma9}}]
First, with $u_{i}(t)=\phi_{i}(t)-y_{i}(t)-\beta k_{i}(t)$, we rewrite (\ref{algorithm1_eq1b}) as
\begin{align*}
x_{i}(t+1)=\sum\limits_{j=1}^{N}A_{ij}(t)x_{j}(t)+u_{i}(t),
\end{align*}
which corresponds to the form formulated in \cite[Lemma 6]{Li2019TAC}. Thus, from the proof of \cite[Lemma 6]{Li2019TAC}, we can directly obtain
\begingroup
\allowdisplaybreaks
\begin{align}\label{lemma9_proof_eq1}
\|x_{i}(t)-\bar{x}(t)\|&\leq\sum\limits_{j=1}^{N}NC_{1}\lambda_{1}^{t-1}\|x_{j}(0)-\bar{x}(0)\|
                       +\sum\limits_{s=0}^{t-1}\sum\limits_{j=1}^{N}NC_{1}\lambda_{1}^{t-1-s}\|u_{j}(s)-\hat{u}(s)\|,
\end{align}
\endgroup
where $\hat{u}(t)=\sum\limits_{j=1}^{N}\pi^{j}(t+1)u_{j}(t)$. Then, for the term $\|u_{j}(s)-\hat{u}(s)\|$, we have
\begingroup
\allowdisplaybreaks
\begin{align}\label{lemma9_proof_eq2}
\|u_{j}(s)-\hat{u}(s)\|&\leq\|u_{j}(s)\|+\bigg\|\sum\limits_{j=1}^{N}\pi^{j}(s+1)u_{j}(s)\bigg\|\notag\\
                           &\leq\sum\limits_{j=1}^{N}\|u_{j}(s)\|+\sum\limits_{j=1}^{N}\pi^{j}(s+1)\|u_{j}(s)\|\notag\\
                           &\leq2\sum\limits_{j=1}^{N}\|u_{j}(s)\|,
\end{align}
\endgroup
where the second inequality is deduced from the convex property of the norm function and the last inequality can be obtained from the fact that $0<\pi^{i}(t)<1$ for all $i=1,2,\cdots,N$ and $t\geq0$. Moreover, considering that
\begin{align*}
\|u_{i}(s)\|\leq\|\phi_{i}(s)\|+\|y_{i}(s)\|+\beta\|k_{i}(s)\|,
\end{align*}
we can get from (\ref{lemma9_proof_eq1}) that
\begingroup
\allowdisplaybreaks
\begin{align}\label{lemma9_proof_eq3}
\|x_{i}(t)-\bar{x}(t)\|
&\leq C_{5}\lambda_{1}^{t-1}+2N^{2}C_{1}\sum\limits_{s=1}^{t-1}\lambda_{1}^{t-1-s}\bigg(w_{1}(s)
+\sum\limits_{j=1}^{N}\|y_{j}(s)\|+\sum\limits_{j=1}^{N}\beta\|k_{i}(s)\|\bigg).
\end{align}
\endgroup
Furthermore, observing that
\begingroup
\allowdisplaybreaks
\begin{align}\label{lemma9_proof_eq4}
g_{i}^{+}(v_{i}(t))
&\leq g_{i}^{+}(z_{i}(t))+g_{i}^{+}(v_{i}(t))-g_{i}^{+}(z_{i}(t))\notag\\
&\leq g_{i}^{+}(z_{i}(t))+(\partial g_{i}^{+}(v_{i}(t)))^{T}(v_{i}(t)-z_{i}(t))\notag\\
&\leq g_{i}^{+}(z_{i}(t))+\|\partial g_{i}^{+}(v_{i}(t))\|\|y_{i}(t)\|\notag\\
&\leq g_{i}^{+}(z_{i}(t))+M\|y_{i}(t)\|,
\end{align}
\endgroup
we arrive at
\begingroup
\allowdisplaybreaks
\begin{align}\label{lemma9_proof_eq5}
\|k_{i}(t)\|
&\leq \frac{1}{\|d_{i}(t)\|}g_{i}^{+}(v_{i}(t))\notag\\
&\leq\frac{1}{\underline{M}}(g_{i}^{+}(z_{i}(t))+g_{i}^{+}(v_{i}(t))-g_{i}^{+}(z_{i}(t)))\notag\\
&\leq\frac{1}{\underline{M}}g_{i}^{+}(z_{i}(t))+\frac{M}{\underline{M}}\|y_{i}(t)\|,
\end{align}
\endgroup
which implies that
\begingroup
\allowdisplaybreaks
\begin{align}\label{lemma9_proof_eq6}
&\veq\|x_{i}(t)-\bar{x}(t)\|\notag\\
&\leq C_{5}\lambda_{1}^{t-1}+2N^{2}C_{1}\sum\limits_{s=1}^{t-1}\lambda_{1}^{t-1-s}\bigg(w_{1}(s)
+\Big(1+\frac{\beta M}{\underline{M}}\Big)\sum\limits_{j=1}^{N}\|y_{j}(s)\|+\frac{\beta}{\underline{M}}w_{2}(s)\bigg).
\end{align}
\endgroup
Finally, combining (\ref{cor_eq1}) with (\ref{lemma9_proof_eq6}) produces
\begingroup
\allowdisplaybreaks
\begin{align}\label{lemma9_proof_eq7}
&\veq\|x_{i}(t)-\bar{x}(t)\|\notag\\
&\leq C_{5}\lambda_{1}^{t-1}+C_{6}\sum\limits_{s=0}^{t-1}\lambda_{1}^{t-1-s}w_{1}(s)+C_{7}\sum\limits_{s=0}^{t-1}\lambda_{1}^{t-1-s}w_{2}(s)\notag\\
&\veq+C_{8}\sum\limits_{s=1}^{t-1}\lambda_{1}^{t-1-s}\lambda_{2}^{s-1}+C_{9}\sum\limits_{s=1}^{t-1}\lambda_{1}^{t-1-s}\sum\limits_{r=0}^{s-1}\lambda_{2}^{s-1-r}\alpha(r)
+C_{10}\sum\limits_{s=1}^{t-1}\lambda_{1}^{t-1-s}\alpha(s),
\end{align}
\endgroup
thereby directly completing the proof.
\end{IEEEproof}

\begin{IEEEproof}[\textbf{Proof~of~Lemma~\ref{lemma10}}]
Noticing that
\begingroup
\allowdisplaybreaks
\begin{align}\label{lemma10_proof_eq1}
&\veq\alpha(t+1)\nabla f_{i}(x_{i}(t+1))-\alpha(t)\nabla f_{i}(x_{i}(t))\notag\\
&\leq\alpha(t+1)\nabla f_{i}(x_{i}(t+1))-\alpha(t)\nabla f_{i}(x_{i}(t+1))\notag\\
&\veq+\alpha(t)\nabla f_{i}(x_{i}(t+1))-\alpha(t)\nabla f_{i}(x_{i}(t)),
\end{align}
\endgroup
under Assumption~\ref{asm1}, we have
\begingroup
\allowdisplaybreaks
\begin{align}\label{lemma10_proof_eq2}
\|\epsilon_{i}(t)\|&\leq M(\alpha(t)-\alpha(t+1))+\alpha(t)\|\nabla f_{i}(x_{i}(t+1))-\nabla f_{i}(x_{i}(t))\|\notag\\
                   &\leq M(\alpha(t)-\alpha(t+1))+\alpha(t)L\|x_{i}(t+1)-x_{i}(t)\|.
\end{align}
\endgroup
According to part (a) of Lemma~\ref{lemma1} and the iteration rules (\ref{algorithm1_eq1a})--(\ref{algorithm1_eq1b}), we get
\begingroup
\allowdisplaybreaks
\begin{align}\label{lemma10_proof_eq3}
&\veq\|x_{i}(t+1)-x_{i}(t)\|\notag\\
&\leq\|v_{i}(t)-\beta k_{i}(t)-x_{i}(t)\|\notag\\
&\leq \|z_{i}(t)-\bar{x}(t)\|+\|x_{i}(t)-\bar{x}(t)\|
+\|y_{i}(t)\|+\beta\|k_{i}(t)\|\notag\\
&\leq \sum\limits_{j=1}^{N}A_{ij}(t)\|x_{j}(t)-\bar{x}(t)\|+\|x_{i}(t)-\bar{x}(t)\|+\|y_{i}(t)\|+\beta\|k_{i}(t)\|,
\end{align}
\endgroup
where the third inequality is obtained from the convex property of the norm function. Noticing (\ref{lemma9_proof_eq5}) and $0\leq A_{ij}(t)<1$ for all $i,j=1,2,\cdots,N$ and $t\geq0$, we can also obtain
\begingroup
\allowdisplaybreaks
\begin{align}\label{lemma10_proof_eq4}
\|\epsilon_{i}(t)\|
&\leq M(\alpha(t)-\alpha(t+1))+L\alpha(t)\sum\limits_{j=1}^{N}\|x_{j}(t)-\bar{x}(t)\|+L\alpha(t)\|x_{i}(t)-\bar{x}(t)\|\notag\\
&\veq+L\Big(1+\frac{\beta M}{\underline{M}}\Big)\alpha(t)\|y_{i}(t)\|+\frac{\beta L}{\underline{M}}\alpha(t)g_{i}^{+}(z_{i}(t)).
\end{align}
\endgroup
Then, based on Corollary~\ref{cor1} and Lemma~\ref{lemma9}, we finally reach
\begingroup
\allowdisplaybreaks
\begin{align}\label{lemma10_proof_eq5}
\|\epsilon_{i}(t)\|
&\leq M(\alpha(t)-\alpha(t+1))+C_{11}\gamma_{1}(t)+C_{12}\gamma_{2}(t)+L(1+N)\xi_{1}(t)+C_{13}\alpha(t)\lambda_{2}^{t-1}\notag\\
&\veq+C_{14}\alpha(t)\sum\limits_{s=0}^{t-1}\lambda_{2}^{t-1-s}\alpha(s)+C_{15}\alpha^{2}(t)+C_{16}\alpha(t)w_{2}(t).
\end{align}
\endgroup
Therefore, we can derive from (\ref{lemma8_proof_eq2}) that
\begingroup
\allowdisplaybreaks
\begin{align}\label{lemma10_proof_eq6}
&\veq\|y_{i}(t)-\mu^{i}(t-1)\tilde{y}(t)\|\notag\\
&\leq C_{3}\lambda_{2}^{t-1}+C^{\prime}_{4}\sum\limits_{s=0}^{t-1}\sum\limits_{j=1}^{N}\lambda_{2}^{t-1-s}\|\epsilon_{j}(s)\|\notag\\
&\leq C_{3}\lambda_{2}^{t-1}+C_{17}\sum\limits_{s=0}^{t-1}\lambda_{2}^{t-1-s}\xi_{2}(s)+C_{18}\eta_{1}(t)+C_{19}\eta_{2}(t)+C_{20}\eta_{3}(t).
\end{align}
\endgroup
So far, we have completed the proof.
\end{IEEEproof}

\begin{remark}
In (\ref{lemma10_proof_eq3}), for the purpose of dealing with the term $\|x_{i}(t+1)-x_{i}(t)\|$, the property $x_{i}(t)\in X_{i}$ is taken into account. It can be observed that the property $x_{i}(t)\in X_{i}$ holds for all $t\geq1$. Thus, although it does not really matter, we can select the initial vales $x_{i}(0)\in X_{i}$ to assure the preciseness of the whole analysis.
\end{remark}

\begin{IEEEproof}[\textbf{Proof~of~Lemma~\ref{lemma11}}]
From the iteration rules (\ref{algorithm1_eq1a})--(\ref{algorithm1_eq1b}), and Lemmas~\ref{lemma1} and~\ref{lemma6}, we can get
\begingroup
\allowdisplaybreaks
\begin{align}\label{lemma11_proof_eq1}
\|x_{i}(t+1)-v\|^{2}
&\leq \|z_{i}(t)-y_{i}(t)-\beta k_{i}(t)-v\|^{2}-\|\phi_{i}(t)\|^{2}\notag\\
&\leq \|z_{i}(t)-y_{i}(t)-v\|^{2}-\beta(2-\beta)\frac{(g_{i}^{+}(v_{i}(t)))^{2}}{\|d_{i}(t)\|^{2}}-\|\phi_{i}(t)\|^{2}\notag\\
&\leq \|z_{i}(t)-y_{i}(t)-v\|^{2}-\frac{\beta(2-\beta)}{M^{2}}(g_{i}^{+}(v_{i}(t)))^{2}-\|\phi_{i}(t)\|^{2}.
\end{align}
\endgroup
The first term on the right-hand side of (\ref{lemma11_proof_eq1}) equals
\begin{equation}\label{lemma11_proof_eq2}
\|z_{i}(t)-v\|^{2}+2y_{i}^{T}(t)(v-z_{i}(t))+\|y_{i}(t)\|^{2}.
\end{equation}
Letting $g(t)=\alpha(t)\sum\limits_{i=1}^{N}\nabla f_{i}(\bar{x}(t))$, for the second term in (\ref{lemma11_proof_eq2}), we can attain
\begingroup
\allowdisplaybreaks
\begin{align}\label{lemma11_proof_eq3}
&\veq 2y_{i}^{T}(t)(v-z_{i}(t))\notag\\
&=2\mu^{i}(t-1)g^{T}(t)(v-\bar{x}(t))+2\mu^{i}(t-1)g^{T}(t)(\bar{x}(t)-z_{i}(t))\notag\\
&\veq+2\mu^{i}(t-1)(\tilde{y}(t)-g(t))^{T}(v-z_{i}(t))+2(y_{i}(t)-\mu^{i}(t-1)\tilde{y}(t))^{T}(v-z_{i}(t)).
\end{align}
\endgroup
Under Assumption~\ref{asm2}, we have that $z_{i}(t)$ is uniformly bounded with respect to $t$, which implies the uniform boundedness of $v-z_{i}(t)$. So we also assume that $\|v-z_{i}(t)\|\leq M$, where $M$ is as given in Remark~\ref{remark3}. Then, under Assumption~\ref{asm1}, it follows from (\ref{lemma11_proof_eq3}) that
\begingroup
\allowdisplaybreaks
\begin{align}\label{lemma11_proof_eq5}
&\veq2y_{i}^{T}(t)(v-z_{i}(t))\notag\\
&\leq2\alpha(t)\mu^{i}(t-1)(f(v)-f(\bar{x}(t)))+2MN\sum\limits_{j=1}^{N}\alpha(t)\|x_{j}(t)-\bar{x}(t)\|\notag\\
&\veq+2LM\sum\limits_{j=1}^{N}\alpha(t)\|x_{j}(t)-\bar{x}(t)\|+2M\|y_{i}(t)-\mu^{i}(t-1)\tilde{y}(t)\|.
\end{align}
\endgroup
Moreover, noticing Proposition~\ref{prop2}, we have that
\begin{equation*}
\text{dist}(x,X_{0})\leq R\max\limits_{1\leq i\leq N}\text{dist}(x,X_{i})+R\max\limits_{1\leq i\leq N}g_{i}^{+}(x),
\end{equation*}
which further implies that
\begingroup
\allowdisplaybreaks
\begin{align}\label{lemma11_proof_eq6}
\|s(t)-\bar{x}(t)\|
&\leq R\sum\limits_{i=1}^{N}\text{dist}(\bar{x},X_{i})+R\sum\limits_{i=1}^{N}g_{i}^{+}(\bar{x})\notag\\
&\leq R\sum\limits_{i=1}^{N}\|x_{i}(t)-\bar{x}(t)\|+R\sum\limits_{i=1}^{N}g_{i}^{+}(z_{i}(t))
+R\sum\limits_{i=1}^{N}g_{i}^{+}(\bar{x})-R\sum\limits_{i=1}^{N}g_{i}^{+}(z_{i}(t))\notag\\
&\leq R(MN+1)\sum\limits_{i=1}^{N}\|x_{i}(t)-\bar{x}(t)\|+R\sum\limits_{i=1}^{N}g_{i}^{+}(z_{i}(t)).
\end{align}
\endgroup
Observing that
\begin{equation*}
f(v)-f(\bar{x}(t))\leq f(v)-f(s(t))+MN\|s(t)-\bar{x}(t)\|,
\end{equation*}
we can get from (\ref{lemma11_proof_eq5}) and (\ref{lemma11_proof_eq6}) that
\begingroup
\allowdisplaybreaks
\begin{align}\label{lemma11_proof_eq7}
2y_{i}^{T}(t)(v-z_{i}(t))
&\leq2\alpha(t)\mu^{i}(t-1)(f(v)-f(s(t)))+2D_{1}\sum\limits_{j=1}^{N}\alpha(t)\|x_{j}(t)-\bar{x}(t)\|\notag\\
&\veq+2MNR\alpha(t)\sum\limits_{i=1}^{N}g_{i}^{+}(z_{i}(t))+2M\|y_{i}(t)-\mu^{i}(t-1)\tilde{y}(t)\|.
\end{align}
\endgroup
Now, for the third term in (\ref{lemma11_proof_eq2}), we can obtain directly from (\ref{lemma8_eq1}) and Lemma~\ref{lemma1} that $\lim\limits_{t\to\infty}\big(y_{i}(t)-\mu^{i}(t-1)\tilde{y}(t)\big)=0$.
So we can also assume that $\|y_{i}(t)-\mu^{i}(t-1)\tilde{y}(t)\|\leq M$ always holds. Then, it follows from Lemma~\ref{lemma7} that
\begingroup
\allowdisplaybreaks
\begin{align}\label{lemma11_proof_eq8}
\|y_{i}(t)\|^{2}&\leq2\|y_{i}(t)-\mu^{i}(t-1)\tilde{y}(t)\|^{2}+2\|\mu^{i}(t-1)\tilde{y}(t)\|^{2}\notag\\
                &\leq2M\|y_{i}(t)-\mu^{i}(t-1)\tilde{y}(t)\|+2M^{2}N^{2}\alpha^{2}(t).
\end{align}
\endgroup
Noting from the convex property of the norm function that
\begin{equation}\label{lemma11_proof_eq9}
\|z_{i}(t)-v\|^{2}\leq\sum\limits_{j=1}^{N}A_{ij}(t)\|x_{i}(t)-v\|^{2},
\end{equation}
we can derive from (\ref{lemma11_proof_eq6}), (\ref{lemma11_proof_eq7}) and (\ref{lemma11_proof_eq8}) that
\begingroup
\allowdisplaybreaks
\begin{align}\label{lemma11_proof_eq10}
&\veq\|z_{i}(t)-y_{i}(t)-v\|^{2}\notag\\
&\leq\sum\limits_{j=1}^{N}A_{ij}(t)\|x_{i}(t)-v\|^{2}+2\alpha(t)\mu^{i}(t-1)(f(v)-f(s(t)))\notag\\
&\veq+2D_{1}\alpha(t)\sum\limits_{j=1}^{N}\|x_{i}(t)-\bar{x}(t)\|+4M^{2}N^{2}R^{2}\tau_{1}\alpha^{2}(t)+\frac{N}{4\tau_{1}}\sum\limits_{i=1}^{N}(g_{i}^{+}(z_{i}(t)))^{2}\notag\\
&\veq+2M^{2}N^{2}\alpha^{2}(t)+4M\|y_{i}(t)-\mu^{i}(t-1)\tilde{y}(t)\|,
\end{align}
\endgroup
with $\tau_{1}>0$ being a constant defined in the following. We next relate $(g_{i}^{+}(v_{i}(t)))^{2}$ to $(g_{i}^{+}(z_{i}(t)))^{2}$. Clearly,
\begin{align*}
g_{i}^{+}(v_{i}(t))=(g_{i}^{+}(v_{i}(t))-g_{i}^{+}(z_{i}(t)))+g_{i}^{+}(z_{i}(t)).
\end{align*}
Thus, we can obtain
\begingroup
\allowdisplaybreaks
\begin{align}\label{lemma11_proof_eq11}
(g_{i}^{+}(v_{i}(t)))^{2}
&\geq2\big(g_{i}^{+}(v_{i}(t))-g_{i}^{+}(z_{i}(t))\big)g_{i}^{+}(z_{i}(t))+(g_{i}^{+}(z_{i}(t)))^{2}\notag\\
&\geq2(\partial g_{i}^{+}(z_{i}(t)))^{T}(v_{i}(t)-z_{i}(t))g_{i}^{+}(z_{i}(t))+(g_{i}^{+}(z_{i}(t)))^{2}\notag\\
&\geq-2M\|y_{i}(t)\|g_{i}^{+}(z_{i}(t))+(g_{i}^{+}(z_{i}(t)))^{2},
\end{align}
\endgroup
where we have also assumed that $\|\partial g_{i}^{+}(z_{i}(t))\|\leq M$. Additionally, observing that
\begingroup
\allowdisplaybreaks
\begin{align}\label{lemma11_proof_eq12}
&\veq-2M\|y_{i}(t)\|g_{i}^{+}(z_{i}(t))\notag\\
&\geq-\tau_{2} M^{2}\|y_{i}(t)\|^{2}-\frac{1}{\tau_{2}}(g_{i}^{+}(z_{i}(t)))^{2}\notag\\
&\geq-2\tau_{2} M^{3}\|y_{i}(t)-\mu^{i}(t-1)\tilde{y}(t)\|-2\tau_{2} M^{4}N^{2}\alpha^{2}(t)-\frac{1}{\tau_{2}}(g_{i}^{+}(z_{i}(t)))^{2},
\end{align}
\endgroup
with $\tau_{2}>0$ being a constant defined in the following, we further have
\begingroup
\allowdisplaybreaks
\begin{align}\label{lemma11_proof_eq13}
&\veq(g_{i}^{+}(v_{i}(t)))^{2}\notag\\
&\geq-2\tau_{2} M^{3}\|y_{i}(t)-\mu_{i}(t-1)\tilde{y}(t)\|-2\tau_{2} M^{4}N^{2}\alpha^{2}(t)+\bigg(1-\frac{1}{\tau_{2}}\bigg)(g_{i}^{+}(z_{i}(t)))^{2}.
\end{align}
\endgroup
Selecting $\tau_{2}=4$, we can get from (\ref{lemma11_proof_eq1}), (\ref{lemma11_proof_eq10}) and (\ref{lemma11_proof_eq12}) that
\begingroup
\allowdisplaybreaks
\begin{align}\label{lemma11_proof_eq14}
&\veq\|x_{i}(t+1)-v\|^{2}\notag\\
&\leq\sum\limits_{j=1}^{N}A_{ij}(t)\|x_{i}(t)-v\|^{2}+2\theta_{2}\alpha(t)(f(v)-f(s(t)))
+2D_{1}\sum\limits_{j=1}^{N}\alpha(t)\|x_{i}(t)-\bar{x}(t)\|\notag\\
&\veq+4M^{2}N^{2}R^{2}\alpha^{2}(t)\tau_{1}+\frac{N}{4\tau_{1}}\sum\limits_{i=1}^{N}(g_{i}^{+}(z_{i}(t)))^{2}+8M^{2}N^{2}\beta(2-\beta)\alpha^{2}(t)+2M^{2}N^{2}\alpha^{2}(t)\notag\\
&\veq+(2M+8M\beta(2-\beta))\|y_{i}(t)-\mu^{i}(t-1)\tilde{y}(t)\|-\frac{3\beta(2-\beta)}{4M^{2}}(g_{i}^{+}(z_{i}(t)))^{2}-\|\phi_{i}(t)\|^{2},
\end{align}
\endgroup
where the first inequality is obtained from the facts that $\theta_{2}\leq\mu^{i}(t)<1$ and $f(v)-f(s(t))\leq0$ for all $i=1,2,\cdots N$ and $t\geq0$. Additionally, from part (c) of Lemma~\ref{lemma4}, we can attain
\begingroup
\allowdisplaybreaks
\begin{align*}
\veq\sum\limits_{i=1}^{N}\!\pi^{i}(t\!+\!1)\!\sum\limits_{j=1}^{N}\!A_{ij}(t)\|x_{i}(t)\!-\!v\|^{2}
\!=\!\sum\limits_{i=1}^{N}\!\pi^{i}(t)\|x_{i}(t)\!-\!v\|^{2}.
\end{align*}
\endgroup
Moreover, from parts (a) and (b) of Lemma~\ref{lemma4}, we can deduce
\begin{align*}
\theta_{1}\leq\pi^{i}(t)<1,\quad \sum\limits_{i=1}^{N}\pi^{i}(t)=1,
\end{align*}
for all $i=1,2,\cdots N$ and $t\geq0$.
Then, selecting $\tau_{1}=\frac{M^{2}N}{\theta_{1}\beta(2-\beta)}$, multiplying both sides of (\ref{lemma11_proof_eq14}) by $\pi^{i}(t+1)$, and taking summation from $i=1$ to $i=N$ yield
\begingroup
\allowdisplaybreaks
\begin{align}\label{lemma11_proof_eq16}
&\veq\sum\limits_{i=1}^{N}\pi^{i}(t+1)\|x_{i}(t+1)-v\|^{2}\notag\\
&\leq\sum\limits_{i=1}^{N}\pi^{i}(t)\|x_{i}(t)-v\|^{2}+2\theta_{2}\alpha(t)(f(v)-f(s(t)))
+2D_{1}\alpha(t)\sum\limits_{j=1}^{N}\|x_{j}(t)-\bar{x}(t)\|\notag\\
&\veq+D_{2}\alpha^{2}(t)+D_{3}\sum\limits_{i=1}^{N}\|y_{i}(t)-\mu_{i}(t-1)\tilde{y}(t)\|\notag\\
&\veq-\sum\limits_{i=1}^{N}\theta_{1}\|\phi_{i}(t)\|^{2}
-\sum\limits_{i=1}^{N}\frac{\theta_{1}\beta(2-\beta)}{2 M^{2}}(g_{i}^{+}(z_{i}(t)))^{2}.
\end{align}
\endgroup
Lastly, combining (\ref{lemma9_eq1}) and (\ref{lemma10_eq1}) with (\ref{lemma11_proof_eq16}) together can directly complete the proof.
\end{IEEEproof}

\begin{IEEEproof}[\textbf{Proof~of~Lemma~\ref{lemma12}}]
From (\ref{lemma11_eq1}) and Proposition~\ref{pro1}, we arrive at
\begingroup
\allowdisplaybreaks
\begin{align}\label{lemma12_proof_eq1}
&\veq\sum\limits_{i=1}^{N}\pi^{i}(t+1)\|x_{i}(t+1)-v\|^{2}+e(t+1)\notag\\
&\leq\sum\limits_{i=1}^{N}\pi^{i}(t)\|x_{i}(t)-v\|^{2}+2\theta_{2}\alpha(t)(f(v)-f(s(t))
+D_{4}\gamma_{1}(t)+D_{5}\gamma_{2}(t)+D_{6}\eta_{1}(t)\notag\\
&\veq-\theta_{1}\sum\limits_{i=1}^{N}\|\phi_{i}(t)\|^{2}+D_{7}\eta_{2}(t)+D_{8}\eta_{3}(t)
-\frac{\theta_{1}\beta(2-\beta)}{2 M^{2}}\sum\limits_{i=1}^{N}(g_{i}^{+}(z_{i}(t)))^{2}+\xi_{3}(t)\notag\\
&\veq+e(t)+a_{1}b_{1}(\lambda_{1}-1)\gamma_{1}(t)+a_{2}b_{2}(\lambda_{1}-1)\gamma_{2}(t)
+c_{1}(\lambda_{2}-1)\eta_{1}(t)\notag\\
&\veq+c_{2}(\lambda_{2}-1)\eta_{2}(t)+a_{3}b_{3}(\lambda_{2}-1)\eta_{3}(t)+a_{1}b_{1}\alpha(t)w_{1}(t)+a_{2}b_{2}\alpha(t)w_{2}(t)\notag\\
&\veq+a_{3}b_{3}\alpha(t)w_{2}(t)+c_{1}\gamma_{1}(t)+c_{2}\gamma_{2}(t).
\end{align}
\endgroup
Clearly, with the given $a_{i},\ i=1,2,3,$ and $c_{i},\ i=1,2$, we have
\begingroup
\allowdisplaybreaks
\begin{align*}
D_{4}\gamma_{1}(t)+a_{1}b_{1}(\lambda_{1}-1)\gamma_{1}(t)+c_{1}\gamma_{1}(t)&=0,\\
D_{5}\gamma_{2}(t)+a_{2}b_{2}(\lambda_{1}-1)\gamma_{2}(t)+c_{2}\gamma_{2}(t)&=0,\\
D_{6}\eta_{1}(t)+c_{1}(\lambda_{1}-1)\eta_{1}(t)&=0,\\
D_{7}\eta_{2}(t)+c_{2}(\lambda_{2}-1)\eta_{2}(t)&=0,\\
D_{8}\eta_{3}(t)+a_{3}b_{3}(\lambda_{2}-1)\eta_{3}(t)&=0.
\end{align*}
\endgroup
Furthermore, noting that
\begingroup
\allowdisplaybreaks
\begin{align*}
a_{1}b_{1}\alpha(t)w_{1}(t)
&\leq \frac{1}{2}a_{1}^{2}\alpha^{2}+\frac{1}{2}b_{1}^{2}w_{1}^{2}(t)\leq \frac{1}{2}a_{1}^{2}\alpha^{2}+\frac{1}{2}b_{1}^{2}N\sum\limits_{i=1}^{N}\|\phi_{i}(t)\|^{2},\\
a_{j}b_{j}\alpha(t)w_{2}(t)
&\leq \frac{1}{2}a_{j}^{2}\alpha^{2}+\frac{1}{2}b_{j}^{2}N\sum\limits_{i=1}^{N}(g_{i}^{+}(z_{i}(t)))^{2},\ j=2,3,
\end{align*}
\endgroup
we can get from (\ref{lemma12_proof_eq1}) that
\begingroup
\allowdisplaybreaks
\begin{align}\label{lemma12_proof_eq2}
&\veq\sum\limits_{i=1}^{N}\pi^{i}(t+1)\|x_{i}(t+1)-v\|^{2}+e(t+1)\notag\\
&\leq\sum\limits_{i=1}^{N}\pi^{i}(t)\|x_{i}(t)-v\|^{2}+e(t)+2\theta_{2}\alpha(t)(f(v)-f(s(t)))\notag\\
&\veq+\xi_{4}(t)-\frac{\theta_{1}}{2}\sum\limits_{i=1}^{N}\|\phi_{i}(t)\|^{2}-\frac{\theta_{1}\beta(2-\beta)}{4 M^{2}}\sum\limits_{i=1}^{N}(g_{i}^{+}(z_{i}(t)))^{2},
\end{align}
\endgroup
thereby completing the proof.
\end{IEEEproof}


\begin{thebibliography}{00}

\bibitem{Wang2011CDC}J. Wang and N. Elia, ``A control perspective for centralized and distributed convex optimization," in \textit{Proc. 50th IEEE Conf. Decision Control}, Orlando, FL, USA, Dec. 2011, pp.~3800--3805.

\bibitem{Gharesifard2014TAC}B. Gharesifard and J. Cort\'{e}s, ``Distributed continuous-time convex optimization on weight-balanced digraphs," \textit{IEEE Trans. Autom. Control}, vol.~59, no.~3, pp.~781--786, 2014.

\bibitem{Kia2015Aut}S. S. Kia, J. Cort\'{e}s, and S. Mart\'{i}nez, ``Distributed convex optimization via continuous-time coordination algorithms with discrete-time communication," \textit{Automatica}, vol.~55, no.~1, pp.~254--264, 2015.


\bibitem{Liang2019Aut}S. Liang, L. Y. Wang, and G. Yin, ``Exponential convergence of distributed primal-dual convex optimization algorithm without strong convexity," \textit{Automatica}, vol.~105, no.~7, pp.~298--306, 2019.

\bibitem{Liu2015TAC}Q. Liu and J. Wang, ``A second-order multi-agent network for bound-constrained distributed optimization," \textit{IEEE Trans. Autom. Control}, vol.~60, no.~12, pp.~3310--3315, 2015.

\bibitem{Shi2013TAC}G. Shi, K. H. Johansson, and Y. Hong, ``Reaching an optimal consensus: Dynamical systems that compute intersections of convex sets," \textit{IEEE Trans. Autom. Control}, vol.~58, no.~3, pp.~610--622, 2013.

\bibitem{Qiu2016Aut}Z. Qiu, S. Yang, and J. Wang, ``Distributed constrained optimal consensus of multi-agent systems," \textit{Automatica}, vol.~68, pp.~209--215, 2016.

\bibitem{Lin2012CDC}P. Lin and W. Ren, ``Distributed subgradient projection algorithm for multi-agent optimization with nonidentical constraints and switching topologies," in \textit{Proc. 51st IEEE Conf. Decision Control}, Maui, HI, USA, Dec. 2012, vol.~5, pp.~6813--6818.

\bibitem{Lin2017TAC}P. Lin, W. Ren, and J. A. Farrell, ``Distributed continuous-time optimization: Nonuniform gradient gains, finite-time convergence, and convex constraint set," \textit{IEEE Trans. Autom. Control}, vol.~62, no.~5, pp.~2239--2253, 2017.

\bibitem{Yan2014NN}Z. Yan, J. Wang, and G. Li, ``A collective neurodynamic optimization approach to bound-constrained nonconvex optimization," \textit{Neural Netw.}, vol.~55, pp.~20--29, 2014.

\bibitem{Yi2014CCC}P. Yi and Y. Hong, ``Distributed continuous-time gradient-based algorithm for constrained optimization," in \textit{Proc. 33rd Chinese Control Conf.}, Nanjing, China, Jul. 2014, pp.~1563--1567.

\bibitem{Yi2015SCL}P. Yi, Y. Hong, and F. Liu, ``Distributed gradient algorithm for constrained optimization with application to load sharing in power systems," \textit{Syst. Control Lett.}, vol.~83, pp.~45--52, 2015.

\bibitem{Liu2017TNNLS}Q. Liu, S. Yang, and J. Wang, ``A collective neurodynamic approach to distributed constrained optimization," \textit{IEEE Trans. Neural. Netw. Learn. Syst.}, vol.~28, no.~8, pp.~1747--1758, 2017.

\bibitem{Liu2013TNNLS}Q. Liu and J. Wang, ``A one-layer projection neural network for nonsmooth optimization subject to linear equalities and bound constraints," \textit{IEEE Trans. Neural Netw. Learn. Syst.}, vol.~24, no.~5, pp.~812--824, 2013.

\bibitem{Yang2017TAC}S. Yang, Q. Liu, and J. Wang, ``A multi-agent system with a proportional-integral protocol for distributed constrained optimization," \textit{IEEE Trans. Autom. Control}, vol.~62, no.~7, pp.~3461--3467, 2017.


\bibitem{Nedic2009TAC}A. Nedi\'{c} and A. Ozdaglar, ``Distributed subgradient methods for multi-agent optimization," \textit{IEEE Trans. Autom. Control}, vol.~54, no.~1, pp.~48--61, 2009.

\bibitem{Nedic2010TAC}A. Nedi\'{c}, A. Ozdaglar, and P. A. Parrilo, ``Constrained consensus and optimization in multi-agent networks," \textit{IEEE Trans. Autom. Control}, vol.~55, no.~4, pp.~922--938, 2010.

\bibitem{Yuan2016SIAM}D. Yuan, D. W. C. Ho, and Y. Hong, ``On Convergence Rate of Distributed Stochastic Gradient Algorithm for Convex Optimization with Inequality Constraints," \textit{SIAM J. Optim.}, vol.~54, no.~5, pp.~2872--2892, 2016.

\bibitem{Shi2017SIAM}W. Shi, Q. Ling, G. Wu, and W. Yin, ``Extra: An exact first-order algorithm for decentralized consensus optimization," \textit{SIAM J. Optim.}, vol.~25, no.~2, pp.~944--966, 2015.



\bibitem{Lei2016SCL}J. Lei, H.-F. Chen, and H.-T. Fang, ``Primal-dual algorithm for distributed constrained optimization" \textit{Syst. Control Lett.}, vol.~96, pp.~110--117, 2016.

\bibitem{Liu2017TAC}Q. Liu, S. Yang, and Y. Hong, ``Constrained consensus algorithms with fixed step size for distributed convex optimization over multiagent networks," \textit{IEEE Trans. Autom. Control}, vol.~62, no.~8, pp.~4259--4265, 2017.


\bibitem{Nedic2015TAC}A. Nedi\'{c} and A. Olshevsky, ``Distributed optimization over time-varying directed graphs," \textit{IEEE Trans. Autom. Control}, vol.~60, no.~3, pp.~601--615, 2015.

\bibitem{Nedic2017SIAM}A. Nedi\'{c}, A. Ozdaglar, and W. Shi, ``Achieving geometric convergence for distributed optimization over time-varying graphs," \textit{SIAM J. Optim.}, vol.~27, no.~4, pp.~2597--2633, 2017.

\bibitem{Liang2020TAC}S. Liang, L. Y. Wang, and G. Yin, ``Dual averaging push for distributed convex optimization over time-varying directed graph," \textit{IEEE Trans. Autom. Control}, vol.~65, no.~4, pp.~1785--1791, 2020.


\bibitem{Gu2020NA}C. Gu, Z. Wu, and J. Li, ``Regularized dual gradient distributed method for constrained convex optimization over unbalanced directed graphs," \textit{Numer. Algorithms}, vol.~84, no.~1, pp.~91--115, 2020.


\bibitem{Pu2018CDC}S. Pu, W. Shi, J. Xu, and A. Nedi\'{c}, ``A push-pull gradient method for distributed optimization in networks," in \textit{Proc. 57th IEEE Conf. Decision Control}, Miami, FL, Dec. 2018, pp.~3385--3390.

\bibitem{Xin2015SCL}R. Xin and U. A. Khan, ``A linear algorithm for optimization over directed graphs with geometric convergence," \textit{Syst. Control Lett.}, vol.~2, no.~3, pp.~315--320, 2018.

\bibitem{Xin2020TAC}R. Xin and U. A. Khan, ``Distributed heavy-ball: A generalization and acceleration of first-order methods with gradient tracking," \textit{IEEE Trans. Autom. Control}, vol.~65, no.~6, pp.~2627--2633, 2020.

\bibitem{Saadatniaki2018}F. Saadatniaki, R. Xin, and U. A. Khan, ``Decentralized optimization over time-varying directed graphs with row and column-stochastic matrices," \textit{IEEE Trans. Autom. Control}, vol.~65, no.~11, pp.~4769--4780, 2020.

\bibitem{Mai2016ACC}V. S. Mai and E. H. Abed, ``Distributed optimization over weighted directed graphs using row stochastic matrix," \textit{Proc. Amer. Control Conf.}, Boston, MA, Jul. 2016, pp.~7165--7170.

\bibitem{Mai2019Aut}V. S. Mai and E. H. Abed, ``Distributed optimization over directed graphs with row stochasticity and constraint regularity," \textit{Automatica}, vol.~102, pp.~94--104, 2019.

\bibitem{Li2019TAC}H. Li, Q. Lv, and T. Huang, ``Distributed projection subgradient algorithm over time-varying general unbalanced directed graphs," \textit{IEEE Trans. Autom. Control}, vol.~64, no.~3, pp.~1309--1316, 2019.


\bibitem{Polyak1969}B. T. Polyak, ``Minimization of nonsmooth functionals," \textit{U.S.S.R. Comput. Math. Math. Phys.}, vol.~9, no.~3, pp.~509--521, 1969.

\bibitem{Nedic2011MP}A. Nedi\'{c}, ``Random algorithms for convex minimization problems," \textit{Math. Program.}, vol.~129, no.~2, pp.~225--253, 2011.


\bibitem{Qu2017TCNS}G. Qu and N. Li, ``Harnessing smoothness to accelerate distributed optimization," \textit{IEEE Trans. Control Netw. Syst.}, vol.~5, no.~3, pp.~1245--1260, 2018.

\bibitem{Xi2018TAC}C. Xi, R. Xin, and U. A. Khan, ``ADD-OPT: Accelerated distributed directed optimization," \textit{IEEE Trans. Autom. Control}, vol.~63, no.~5, pp.~1329--1339, 2018.
%

\bibitem{Xin2020Proc}R. Xin, S. Pu, A. Nedi\'{c}, and U. A. Khan, ``A general framework for decentralized optimization with first-order methods," \textit{Proc. IEEE}, vol.~108, no.~11, pp.~1869--1889, 2020.





%
%

\end{thebibliography}
\end{document}